\documentclass[a4paper,12pt]{article}

\usepackage{color}

\usepackage{makeidx}
\usepackage{bm}
\usepackage{latexsym}
\usepackage{amscd}
\usepackage{amsmath}
\usepackage{amssymb}

\usepackage{stmaryrd}

\usepackage{amsthm}
\usepackage{float}
\usepackage{graphicx}
\usepackage{psfrag}

\theoremstyle{plain}
\newtheorem{theorem}{Theorem}[section]

\newtheorem{lemma}[theorem]{Lemma}

\theoremstyle{definition}
\newtheorem{definition}[theorem]{Definition}
\newtheorem{remark}[theorem]{Remark}
\newtheorem{example}[theorem]{Example}
\newtheorem{proposition}[theorem]{Proposition}

\renewcommand{\Im}{\ensuremath{\mathop{\mathfrak{Im}}\nolimits}}

\newcommand{\id}{\ensuremath{\mathop{\mathrm{id}}}}

\newcommand{\bC}{\ensuremath{\mathbb{C}}}

\newcommand{\bP}{\ensuremath{\mathbb{P}}}

\newcommand{\bR}{\ensuremath{\mathbb{R}}}

\newcommand{\bZ}{\ensuremath{\mathbb{Z}}}

\newcommand{\scA}{\ensuremath{\mathcal{A}}}

\newcommand{\scF}{\ensuremath{\mathcal{F}}}

\newcommand{\scL}{\ensuremath{\mathcal{L}}}

\newcommand{\ch}{\mathsf{ch}}

\newcommand{\ii}{\sqrt{-1}}

\newcommand{\eps}{\varepsilon}
\newcommand{\pv}{\operatorname{pvarg}}
\newcommand{\Sep}{\operatorname{Sep}}
\newcommand{\Ker}{\operatorname{Ker}}
\newcommand{\pr}{\operatorname{pr}}

\title{Minimal stratifications for line arrangements 
and positive homogeneous presentations for fundamental groups}
\author{Masahiko Yoshinaga}
\date{\today}
\begin{document}
\maketitle

\begin{abstract}
The complement of a complex hyperplane arrangement 
is known to be homotopic to a minimal CW complex. 
There are several approaches to the minimality. 
In this paper, we restrict our attention to 
real two dimensional cases, and introduce 
the ``dual'' objects so called minimal 
stratifications. The strata are explicitly 
described as semialgebraic sets. The stratification 
induces 
a partition of the complement into a disjoint 
union of contractible spaces, which is minimal 
in the sense that the number of codimension $k$ 
pieces equals the $k$-th Betti number. 

We also discuss presentations for the fundamental 
group associated to the minimal stratification. 
In particular, we show that the fundamental groups 
of complements of a real arrangements have 
positive homogeneous presentations. 
\end{abstract}

\section{Introduction}
\label{sec:intro}

In 1980s Randell found an algorithm for presenting 
the fundamental group of the complement $M(\scA)$ 
of arrangement $\scA$ of complexified lines in $\bC^2$ 
(\cite{ran-fun, fal-hom}). 
Various algorithms for doing this 
were found subsequently 
(\cite{arv-fun, coh-suc, moi-tei}). 
It was observed that 
these presentations are minimal in the 
sense that the numbers of generators and relations are 
equal to $b_1(\pi_1)$ and $b_2(\pi_1)$, respectively, 
(c.f. $b_i(M)=b_i(\pi_1(M(\scA)))$ for $i\leq 2$ 
\cite{ran-hom}) 
and 
several presentations are homotopic to $M(\scA)$. 
(It is not clear to the author that whether or not 
every minimal presentation is homotopic to 
$M(\scA)$, which is true for braid-monodromy presentation 
\cite{lib}.) 

These works have been partially generalized to higher 
dimensional cases. 
Let $\scA$ be an arrangement of hyperplanes in $\bC^\ell$. 
The complement $M(\scA)=\bC^\ell\setminus\scA$ is proved 
to be homotopic to a minimal CW complex, that is, 
a finite CW complex in which the number of $p$-cells 
equals the $p$-th Betti number \cite{ps-h, dim-pap, ran-mor}. 
The minimality is expected to have applications to 
topological problems of arrangements. 
In order to apply, we need to make explicit how cells 
in the minimal CW complex are attached. 
There are two approaches to describe 
the minimal structure of $M(\scA)$, 
one is based on classical Morse theoretic 
study of Lefschetz's theorem on hyperplane 
section \cite{yos-lef}, the other is based on 
discrete Morse theory of Salvetti complex 
\cite{sal-sett, del-min}. 
There are also 
some applications to computations of local system 
(co-)homology groups \cite{gai-sal, yos-ch, yos-loc}.

The purpose of this paper is to describe 
the ``dual'' object to the minimal CW complex 
for $\ell=2$. 
We introduce the minimal stratification 
$M(\scA)=X_0\supset X_1\supset X_2$ for 
the complement $M(\scA)$ such that 
\begin{itemize}
\item $X_0\setminus X_1=U$ is a contractible $4$-manifolds, 
\item $X_1\setminus X_2=\bigsqcup_{i=1}^{b_1(M)}S_i^\circ$ 
is a disjoint union of contractible $3$-manifolds, 
such that the number of pieces is equal to the 
$1$st Betti number $b_1(M)$, and 
\item $X_2=\bigsqcup_{\lambda=1}^{b_2(M)}C_\lambda$ 
is a disjoint union of 
contractible $2$-manifolds (chambers), such that 
the number of pieces is equal to the $2$nd Betti 
number $b_2(M)$. 
\end{itemize}
(see Theorem \ref{thm:main} for details). 
We describe explicitly the strata as semialgebraic sets. 
For such stratification, we can take generators 
and relations of $\pi_1(M)$ which are dual to 
the strata. By analyzing the incidence relation of 
strata, we obtain a presentation for $\pi_1$ which is not 
rely on braid monodromy or Zariski-van Kampen method. 
The resulting presentation has only positive homogeneous 
relations.

This paper is organized as follows. 
In \S\ref{sec:1dim}, as a motivating example, 
we compare the minimal stratification with 
Morse theoretic description of minimal CW complex 
for a very simple example: two points 
$\{0, 1\}$ in $\bR$. 
In \S\ref{sec:notation} we recall basic facts and 
introduce the sail $S(\alpha, \beta)$ 
bound to lines. The sail is 
a $3$-dimensional semialgebraic 
submanifold of $M(\scA)$ which will be used to 
define the minimal stratification. 
\S\ref{sec:min} contains the main result. 
The proof will be given in 
\S\ref{sec:mainproof}. 
In \S\ref{sec:dualpre} we discuss the presentation 
for $\pi_1(M(\scA))$ associated to the minimal 
stratification. The generators are taken 
as transversal loops to the strata. 
In \S\ref{sec:poshom} we take meridian generators 
for the fundamental group. By computing relations 
in the previous section with respect to the new 
generators, we reach the positive 
homogeneous presentation.

\section{A one-dimensional example}
\label{sec:1dim}

\begin{example}
Let 
$M=\bC\setminus\{0, 1\}$ and 
$\varphi(z):=\frac{(z+1)^2}{\sqrt{z(z-1)}}$. 
We consider $|\varphi|:M\rightarrow\bR$ as a Morse function 
which has three critical points 
$z=-1, \frac{5-\sqrt{17}}{4}, \frac{5+\sqrt{17}}{4}$ with 
index $0, 1, 1$ respectively. Note that all critical points 
are real and 
$0<\frac{5-\sqrt{17}}{4}<1<\frac{5+\sqrt{17}}{4}$. 
The unstable manifolds present a one-dimensional CW complex 
which is homotopic to $M$. 
Since $|\varphi(z)|\rightarrow\infty$ as $|z|\rightarrow\infty$, 
the unstable cells are as in Figure \ref{fig:unst}. 
It is not easy to describe the unstable manifolds explicitly 
even for one-dimensional cases. Nevertheless, the stable manifolds 
can be explicitly described: two open segments $(0,1), (1, \infty)$ 
and the remainder $U=M\setminus((0,1)\cup(1,\infty))$.

\begin{figure}[htbp]
\begin{picture}(100,100)(20,0)
\thicklines

\put(200,50){\circle{5}}
\put(300,50){\circle{5}}

\put(87,35){$-1$}
\put(197,35){$0$}
\put(297,35){$1$}

\put(100,50){\circle*{5}}
\put(220,50){\circle*{5}}
\put(330,50){\circle*{5}}
\put(223,33){$\frac{5-\sqrt{17}}{4}$}
\put(333,33){$\frac{5+\sqrt{17}}{4}$}

\qbezier(220,50)(220,75)(200,75)
\qbezier(200,75)(150,75)(100,50)
\qbezier(220,50)(220,25)(200,25)
\qbezier(200,25)(150,25)(100,50)

\qbezier(330,50)(330,100)(250,100)
\qbezier(250,100)(120,100)(100,50)
\qbezier(330,50)(330,0)(250,0)
\qbezier(250,0)(120,0)(100,50)

\multiput(203,50)(3,0){32}{\circle*{1}}
\multiput(303,50)(3,0){32}{\circle*{1}}

\end{picture}
     \caption{Unstable and stable manifolds (thick and dotted line, 
respectively).}
\label{fig:unst}
\end{figure}
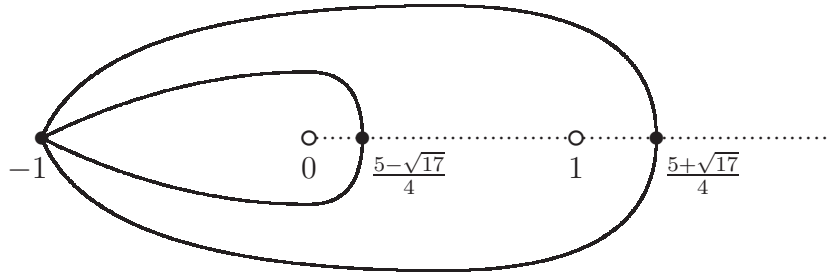

We have a partition 
$U\sqcup (0,1)\sqcup (1,\infty)$ of $M$ by 
contractible pieces, and 
note that the number of codimension zero piece is 
equal to $b_0(M)=1$ and that of codimension one is 
$b_1(M)=2$. 
Also note that codimension one pieces 
$(0,1)$ and $(1, \infty)$ 
are nothing but chambers of the real hyperplane arrangement 
$\{0, 1\}$. These pieces are 
expressed in terms of defining linear forms as follows, 
\begin{equation}
\label{eq:segment}
\begin{split}
(0,1)=\left\{z\in M\left|
\frac{z-1}{z}\in\bR_{<0}
\right.\right\}, 
\\
(1,\infty)=\left\{z\in M\left|
\frac{-1}{z-1}\in\bR_{<0}
\right.\right\}, 
\end{split}
\end{equation}
where $\bR_{<0}$ is the set of negative real numbers. 
\end{example}

The homotopy types of the unstable cells for 
higher dimensional cases are discussed 
in \cite{yos-lef}. 
The unstable cell itself is highly transcendental. 
We will see that the submanifolds defined by 
formulae similar to (\ref{eq:segment}) stratify 
the complement $\bC^2$ minus lines. Also it gives 
a partition into the disjoint union of 
contractible manifolds. 



\section{Basic notation}
\label{sec:notation}

\subsection{Setting}
\label{subsec:setting}

A real arrangement $\scA =\{H_1, \dots, H_n\}$ is a 
finite set of affine lines in the affine plane 
$\bR^2$. Each line is defined by some affine 
linear form 
\begin{equation}
\label{eq:defining}
\alpha_H(x_1, x_2)=ax_1+bx_2+c=0, 
\end{equation}
with $a, b, c\in\bR$ and $(a, b)\neq (0,0)$. 
A connected component of $\bR^2\setminus\bigcup_{H\in\scA}H$ 
is called a chamber. The set of all chambers is denoted by 
$\ch(\scA)$. The affine linear equation (\ref{eq:defining}) 
defines a complex line 
$\{(z_1, z_2)\in\bC^2\mid az_1+bz_2+c=0\}$ in $\bC^2$. 
We denote the set of complexified lines by 
$\scA_\bC=\{H_\bC=H\otimes\bC\mid H\in\scA\}$. 
The object of our interest is the complexified 
complement $M(\scA)=\bC^2\setminus\bigcup_{H\in\scA}H_\bC$. 

\subsection{Generic flags and numbering of lines}

Let $\scF$ be a generic flag in $\bR^2$
$$
\scF:
\emptyset=
\scF^{-1}\subset
\scF^{0}\subset
\scF^{1}\subset
\scF^{2}=\bR^2, 
$$
where $\scF^k$ is a generic $k$-dimensional 
affine subspace. 

\begin{definition}
For $k=0, 1, 2$, define the subset 
$\ch_k^\scF(\scA)\subset\ch(\scA)$ by 
$$
\ch_k^\scF(\scA):=
\{C\in\ch(\scA)\mid C\cap\scF^k\neq\emptyset, 
C\cap\scF^{k-1}=\emptyset\}.
$$ 
\end{definition}
The set of chambers decomposes into a 
disjoint union, 
$\ch(\scA)=
\ch_0^\scF(\scA)\sqcup
\ch_1^\scF(\scA)\sqcup
\ch_2^\scF(\scA)$. The cardinality of 
$\ch_k^\scF(\scA)$ is given as follows, 
which is an application of Zaslawski's 
formula \cite{zas-face}. 

\begin{proposition}
\begin{equation*}
\begin{split}
&\sharp\ch_0^\scF(\scA)=b_0(M(\scA))=1, \\
&\sharp\ch_1^\scF(\scA)=b_1(M(\scA))=n, \\
&\sharp\ch_2^\scF(\scA)=b_2(M(\scA)). 
\end{split}
\end{equation*}
\end{proposition}

\subsection{Assumptions on generic flag and numbering}
\label{subsec:numbering}

Throughout this paper, we assume that 
the generic flag $\scF$ satisfies the following 
conditions: 
\begin{itemize}
\item $\scF^1$ does not separate intersections of $\scA$, 
\item $\scF^0$ does not separate $n$-points 
$\scA\cap\scF^1$. 
\end{itemize}
Then we can choose coordinates $x_1, x_2$ so that 
$\scF^0$ is the origin, 
$\scF^1$ is given by $x_2=0$, all intersections of $\scA$ 
are contained in the upper-half plane $\{(x_1, x_2)\in\bR^2\mid
x_2>0\}$ and $\scA\cap\scF^1$ is contained in the 
half line $\{(x_1, 0)\mid x_1>0\}$. 

We set $H_i\cap\scF^1$ has coordinates $(a_i, 0)$. 
By changing the numbering of lines and signs of 
the defining equation $\alpha_i$ of $H_i\in\scA$ 
we may assume 
\begin{itemize}
\item $0<a_n<a_{n-1}<\dots<a_1$, and 
\item the origin $\scF^0$ is contained in the negative 
half plane $H_i^-=\{\alpha_i<0\}$. 
\end{itemize}
\begin{remark}
Sometimes it is convenient to consider 
$0$-th line $H_0$ to be the line at infinity $H_0$ with 
defining equation $\alpha_0=-1$ and $a_0=+\infty$. 
\end{remark}
We also put 
$\ch_0^\scF(\scA)=\{C_0\}$ and 
$\ch_1^\scF(\scA)=\{C_1, \dots, C_n\}$ so that 
$C_k\cap\scF^1$ is equal to the interval $(a_k, a_{k-1})$. 
(We use the convention $a_0=+\infty$.) 
It is easily seen that the 
chambers $C_0$ and $C_k$ ($k=1, \dots, n$) 
have the following expression. 
\begin{equation}
\begin{split}
&C_0=\bigcap_{i=1}^n\{\alpha_i<0\},\\
&C_k=
\bigcap_{i=0}^{k-1}\{\alpha_i<0\}\cap
\bigcap_{i=k}^n\{\alpha_i>0\},\ (k=1, \dots, n).
\end{split}
\end{equation}
(We consider $\alpha_0<0$ whole $\bR^2$.) 
The notations introduced in this section are illustrated 
in Figure \ref{fig:numbering}.


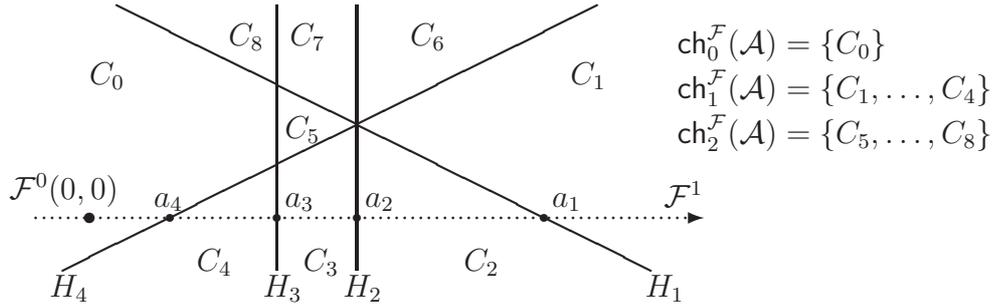
\begin{figure}[htbp]
\begin{picture}(100,100)(20,0)
\thicklines

\put(70,20){\circle*{4}}
\put(40,27){$\scF^0(0,0)$}

\multiput(50,20)(3,0){83}{\circle*{1}}
\put(300,20){\vector(1,0){0}}
\put(285,24){$\scF^1$}

\put(280,0){\line(-2,1){200}}
\put(240,20){\circle*{3}}
\put(243,24){$a_1$}
\put(278,-10){$H_1$}

\put(170,0){\line(0,1){100}}
\put(170,20){\circle*{3}}
\put(173,24){$a_2$}
\put(165,-10){$H_2$}

\put(140,0){\line(0,1){100}}
\put(140,20){\circle*{3}}
\put(143,24){$a_3$}
\put(135,-10){$H_3$}

\put(60,0){\line(2,1){200}}
\put(100,20){\circle*{3}}
\put(94,24){$a_4$}
\put(55,-10){$H_4$}

\put(70,70){$C_0$}
\put(250,70){$C_1$}
\put(210,0){$C_2$}
\put(150,0){$C_3$}
\put(110,0){$C_4$}
\put(143,50){$C_5$}
\put(190,85){$C_6$}
\put(145,85){$C_7$}
\put(122,85){$C_8$}

\put(290,82){$\ch_0^\scF(\scA)=\{C_0\}$}
\put(290,65){$\ch_1^\scF(\scA)=\{C_1,\dots,C_4\}$}
\put(290,48){$\ch_2^\scF(\scA)=\{C_5,\dots,C_8\}$}

\end{picture}
     \caption{Numbering of lines and chambers.}\label{fig:numbering}
\end{figure}


\subsection{Sails bound to lines}
\label{subsec:sail}

Let $\alpha, \beta\in\bC[z_1, z_2]$ be 
polynomials of $\deg\leq 1$. 
We assume that $\alpha\neq 0, \beta\neq 0$ and they are 
linearly independent over $\bC$. 
(Note that we allow the situation that one of $\alpha$ or $\beta$ is 
equal to a non-zero constant.) 

\begin{definition}
\label{def:sail}
For $\alpha$ and $\beta$ as above, we define the 
{\em sail bound to $\alpha$ and $\beta$} by 
$$
S(\alpha. \beta)=
\left\{
z=(z_1, z_2)\in\bC^2
\left|
\alpha(z)\beta(z)\neq 0,\ 
\frac{\alpha(z)}{\beta(z)}\in\bR_{<0}
\right.
\right\}. 
$$
\end{definition}
The sail 
$S(\alpha, \beta)$ is a closed subset of 
$\bC^2\setminus\{\alpha\beta=0\}$. 
Furthermore we have: 

\begin{lemma}
$S(\alpha, \beta)$ is an orientable $3$-dimensional 
manifold. More precisely, 
\begin{itemize}
\item[(1)] 
if $\alpha$ and $\beta$ determine 
intersecting lines, 
then $S(\alpha, \beta)$ is diffeomorphic to 
$\bC^*\times\bR_{<0}$. 
\item[(2)] 
else, (i.e., either 
$\alpha$ and $\beta$ determine parallel 
lines or one of 
$\alpha$ and $\beta$ is a nonzero constant), 
then $S(\alpha, \beta)$ is diffeomorphic to 
$\bC\times\bR_{<0}$. 
\end{itemize}
\end{lemma}

\proof

Case (1): Suppose that $\deg\alpha=\deg\beta=1$ and 
two lines intersects. Then the map 
$$
\begin{array}{cccc}
(\alpha, \beta):&\bC^2&\longrightarrow&\bC^2\\
&&&\\
&z&\longmapsto&(\alpha(z), \beta(z))
\end{array}
$$
is isomorphic. The image of the sail $S(\alpha, \beta)$ by 
the map $(\alpha, \beta)$ is 
$$
\left\{
(s, t)\in\bC^2
\left|
s\cdot t\neq 0,\ 
\frac{s}{t}\in\bR_{<0}
\right.
\right\}, 
$$
where $s, t$ are coordinates of the target $\bC^2$. 
The image is isomorphic to $\bC^*\times\bR_{<0}$ 
by the the isomorphism $(s, t)\mapsto (t, s/t)$ 
of $(\bC^*)^2$. 

\noindent
Case (2):  
Suppose that $\deg\alpha=\deg\beta=1$ and 
two lines are parallel. In this case we may 
assume that $\beta=p\alpha + q$ with $p, q\in\bC^*$. 
Choose another linear equation $\gamma$ such that 
lines $\alpha=0$ and $\gamma=0$ are intersecting. 
Then 
$$
\begin{array}{cccc}
(\alpha, \gamma):&\bC^2&\longrightarrow&\bC^2\\
&&&\\
&z&\longmapsto&(\alpha(z), \gamma(z))
\end{array}
$$
is isomorphic. The image of $S(\alpha, \beta)$ is 
expressed as 
$$
\left\{
(s, t)\in\bC^2
\left|
s\cdot t\neq 0,\ 
\frac{s}{ps+q}\in\bR_{<0}
\right.
\right\}. 
$$
It is easily checked that the set 
$$
\left\{
s\in\bC
\left|
\frac{s}{p(s+\frac{q}{p})}\in\bR_{<0}
\right.
\right\}
$$
is an open arc connecting $0$ and $-\frac{q}{p}\in\bC$. 
Thus $S(\alpha, \beta)$ is isomorphic to the product of 
the open arc and $\bC$. 

\noindent
Case (3): The proof for the case 
$\deg\alpha=1$ and $\deg\beta=0$ is 
similar to the case (2). 
\qed

\subsection{Orientations}
\label{subsec:ori}

For the purpose of obtaining a presentation for 
the fundamental group of $M(\scA)$, 
intersection numbers of loops and sails 
play crucial roles. It is necessary to specify 
the orientation of the sail $S(\alpha, \beta)$.

We first recall that the orientation of 
$\bC^2$ is given by the identification 
$$
\begin{array}{rcl}
\bC^2&\stackrel{\sim}{\longrightarrow}&\bR^4\\
&&\\
(z_1, z_2)&\longmapsto&(x_1, y_1, x_2, y_2), 
\end{array}
$$
where $z_i=x_i+\ii y_i$. 
Consider the map $\varphi=\frac{\alpha}{\beta}: 
\bC^2\setminus\{\alpha\beta=0\}\rightarrow\bC$. 
Since $S(\alpha, \beta)$ is connected, it is 
enough to specify an orientation of 
$T_pS(\alpha, \beta)$ for a point $p\in S(\alpha, \beta)$. 
The following two ordered direct sums 
determine an orientation of $S(\alpha, \beta)$: 
\begin{equation*}
\begin{split}
T_pS(\alpha, \beta)\oplus N_p(S(\alpha, \beta), \bC^2)&=T_p\bC^2\\
T_{\varphi(p)}R_{<0}\oplus \varphi_* N_p(S(\alpha, \beta), \bC^2)&=
T_{\varphi(p)}\bC, 
\end{split}
\end{equation*}
where $N_p(S, \bC^2)$ is a normal bundle. Note that we consider 
the orientation of $\bR_{<0}$ induced from the inclusion 
$\bR_{<0}\subset\bR$. 

\begin{remark}
$S(\alpha, \beta)$ and $S(\beta, \alpha)$ are 
the same as manifolds, but orientations are 
different. 
\end{remark}

The above definition is equivalent to saying as follows. 
Let $c:(-\varepsilon, \varepsilon)\longrightarrow
\bC^2\setminus\{\alpha\beta=0\}$ be a differentiable 
map transversal to $S(\alpha, \beta)$. Assume that 
$c^{-1}(S(\alpha, \beta))=\{0\}$. 
Then $c$ intersects $S(\alpha, \beta)$ positively 
(denoted by $I_{c(0)}(S(\alpha, \beta), c)=+1$) if 
and only if 
$$
\varphi_*(\dot{c}(0))\in T_{\varphi(c(0))}\bC\simeq \bC
$$
has positive imaginary part (Figure \ref{fig:intersection}). 

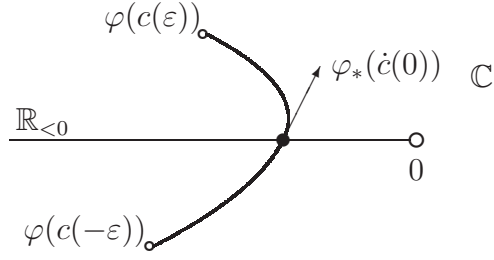
\begin{figure}[htbp]
\begin{picture}(100,100)(20,0)
\thicklines

\put(250,50){\circle{5}}
\put(247,35){$0$}

\thinlines
\put(247.5,50){\line(-1,0){150}}
\put(100,54){$\bR_{<0}$}

\thicklines

\put(270,70){$\bC$}

\qbezier(151.5,10)(190,30)(200,50)
\qbezier(200,50)(210,70)(171.5,90)

\put(150,10){\circle{3}}
\put(170,90){\circle{3}}
\put(103,13){$\varphi(c(-\varepsilon))$}
\put(133,93){$\varphi(c(\varepsilon))$}

\put(200,50){\circle*{5}}
\thinlines
\put(200,50){\vector(1,2){14}}
\put(218,75){$\varphi_*(\dot{c}(0))$}

\end{picture}
\caption{$\varphi\circ c:(-\varepsilon, \varepsilon)
\longrightarrow\bC$.}
\label{fig:intersection}
\end{figure}
Let us look at an example showing how the 
intersection numbers are computed.

\begin{example}
\label{ex:int}
Let $\varphi(z_2, z_1)=\frac{z_2}{z_1}$ and 
$$
S:=S(z_2, z_1)=
\left\{
(z_1, z_2)\in(\bC^*)^2
\mid
\varphi({z_2}, {z_1})\in\bR_{<0}
\right\}. 
$$
Fix positive real numbers $r, \varepsilon>0$ and 
an argument $0\leq\theta_0<2\pi$. 
Consider the continuous map 
$$
\begin{array}{rccl}
\gamma:&\bR/2\pi\bZ&\longrightarrow&(\bC^*)^2\\
&&&\\
&t&\longmapsto&
r(\cos\theta_0, \sin\theta_0)+
\ii\eps
(\cos t, \sin t). 
\end{array}
$$
Then $\gamma(t)\in S$ if and only if 
$\varphi(\gamma(t))=
\frac{r\sin\theta_0+\ii\eps\sin t}{r\cos\theta_0+\ii\eps\cos t}$ 
is a negative real number. 
Since 
$$
\frac{r\sin\theta_0+\ii\eps\sin t}{r\cos\theta_0+\ii\eps\cos t}
=
\frac{r^2\sin\theta_0\cos\theta_0+\eps^2\sin t\cos t+
\ii\cdot r\cdot\eps\sin(t-\theta_0)}
{r^2\cos^2\theta_0+\eps^2\cos^2 t}, 
$$
it is contained in $\bR_{<0}$ if and only if 
$t=\theta_0, \theta_0+\pi$ and 
$\sin\theta_0\cdot\cos\theta_0<0$ (equivalently either 
$\frac{\pi}{2}<\theta_0<\pi$ or 
$\frac{3\pi}{2}<\theta_0<2\pi$). 
In such cases it is easily seen that 
$\Im\varphi_*(\dot{\gamma}(\theta_0))>0$ and 
$\Im\varphi_*(\dot{\gamma}(\theta_0+\pi))<0$. 
Hence we have 
$$
I_{\gamma(\theta_0)}(S, \gamma)=+1, \mbox{ and }
I_{\gamma(\theta_0+\pi)}(S, \gamma)=-1. 
$$
\end{example}

\section{Minimal Stratification}
\label{sec:min}

\subsection{Main result}
\label{subsec:main}

In this section we shall give an explicit stratification 
of the complement $M(\scA)$ by using chambers and sails. 
We keep the notations as in \S\ref{subsec:numbering}. 
First recall that the sail defined by $\alpha_i$ and 
$\alpha_{i-1}$ is 
$$
S(\alpha_{i-1}, \alpha_{i})=
\left\{
z\in\bC^2
\left|
\alpha_{i-1}(z)\cdot\alpha_i(z)\neq 0,\ 
\frac{\alpha_{i-1}(z)}{\alpha_i(z)}\in\bR_{<0}
\right.
\right\}. 
$$
(we use the convention $\alpha_0=-1$). 
Then 
$$
S_i:=S(\alpha_{i-1}, \alpha_{i})\cap M(\scA)
$$
is an oriented $3$-dimensional closed submanifold of 
$M(\scA)$ for $i=1, \dots, n$. These $S_i$'s stratify 
the complement $M(\scA)$. 
\begin{proposition}
Let $C\in\ch(\scA)$ and $i=1, \dots, n$. 
The following are equivalent. 
\begin{itemize}
\item[(a)] 
$C\subset S_i$. 
\item[(b)] 
$C\cap S_i\neq\emptyset$. 
\item[(c)] 
$\alpha_i(C)\cdot\alpha_{i-1}(C)<0$. (We use the 
convention $\alpha_0=-1$.) 
\end{itemize}
\end{proposition}

Now we state the main result. 

\begin{theorem}
\label{thm:main}
The closed submanifolds $S_1, \dots, S_n\subset M(\scA)$ 
satisfy the following. 
\begin{itemize}
\item[(i)] 
$S_i$ and $S_j$ ($i\neq j$) intersect transversely, and 
$S_i\cap S_j=\bigsqcup C$, where $C$ runs all chambers 
satisfying 
$\alpha_i(C)\alpha_{i-1}(C)<0$ and 
$\alpha_j(C)\alpha_{j-1}(C)<0$. 
\item[(ii)] 
$
S_i^\circ:=S(\alpha_i, \alpha_{i-1})\setminus
\bigcup_{C\in\ch_2^\scF(\scA)}C
$
is a contractible $3$-manifold. 
\item[(iii)] 
$U:=M(\scA)\setminus\bigcup_{i=1}^n S_i$ 
is a contractible $4$-manifold. 
\end{itemize}
\end{theorem}
The proof will be given in \S\ref{sec:mainproof}. 

\begin{remark}
Theorem \ref{thm:main} gives rise to a partition of 
$M(\scA)$ into disjoint union of contractible manifolds 
$M(\scA)=U\sqcup\bigsqcup_{i=1}^n S_i^\circ
\sqcup\bigsqcup_{C\in\ch_2^\scF(\scA)}C$. 
Such partitions are obtained in \cite{ito-yos} for any dimension. 
However, the partition $M(\scA)=\sqcup S_\lambda$ 
in \cite{ito-yos} is not induced from a stratification. In other 
words, it does not satisfy the following property: 
$\overline{S_\lambda}\setminus S_\lambda$ is a union of 
other pieces of smaller dimensions. 
We do not know explicit minimal stratification for 
dimension $\geq 3$. 

\end{remark}

\section{Dual presentation for the fundamental group}
\label{sec:dualpre}

Using Theorem \ref{thm:main}, we give a presentation for 
the fundamental group $\pi_1(M(\scA))$. The idea is 
that we take the base point in $U$ and 
transversal loop to each $S_i$ as a generator, then 
relations are generated by loops around chambers 
$C\in\ch_2^\scF(\scA)$. 

\subsection{Transversal generators}
\label{subsec:transvg}

Fix a base point $*\in U$ and a point $p_i\in S_i^\circ$. 
There exists a continuous curve 
$\eta_i:[0,1]\rightarrow M(\scA)$ such that 
\begin{itemize}
\item $\eta_i(0)=\eta_i(1)=*$, 
\item $\eta_i(1/2)=p_i$ and $\eta_i^{-1}(S_i)=\{1/2\}$, 
\item $\eta_i$ intersects $S_i^\circ$ transversely and positively, 
that is, $I_{p_i}(S_i^\circ, \eta_i)=1$, and it does not 
intersect $S_j$ for $j\neq i$. 
\end{itemize}
Since $U$ and $S_i^\circ$ are contractible, 
the homotopy type of 
$\eta_i$ is independent of the choice of $\eta_i$. 

Let $\eta:[0,1]\rightarrow M(\scA)$ be a continuous 
map with $\eta(0), \eta(1)\in U$ (not necessarily 
$\eta(0)=\eta(1)=*$). Since $U$ is contractible, 
there exist paths $c_1$ from the base point $*$ to $\eta(0)$ and 
$c_2$ from $\eta(1)$ to $*$. Then $c_1\eta c_2$ is a loop 
which homotopy class $[c_1\eta c_2]\in\pi_1(M(\scA), *)$ 
is uniquely determined by $\eta$. 
We denote the class by $[\eta]\in\pi_1(M(\scA), *)$ 
for simplicity. 

\begin{lemma}
\label{lem:generate}
With the notation above, 
$[\eta_1], \dots, [\eta_n]$ generate 
$\pi_1(M(\scA), *)$. 
\end{lemma}

\proof
Let $\eta:[0,1]\rightarrow M(\scA)$ be a continuous 
map such that $\eta(0)=\eta(1)=*$. 
By the transversality homotopy theorem 
(e.g., \cite[Chap 2]{gui-pol}), 
we can perturb $\eta$ into a new loop 
such that the following hold: 
\begin{itemize}
\item 
The image of $\eta$ is disjoint from 
$\bigsqcup_{C\in\ch_2^\scF(\scA)}C$. 
\item 
The image of $\eta$ intersects 
$\bigsqcup_{i=1}^nS_i^\circ$ transversely. 
\end{itemize}
Suppose that 
$\eta^{-1}(\bigsqcup_{i=1}^nS_i^\circ)=
\{t_1, \dots, t_N\}$ with 
$0<t_1<\dots<t_N<1$ and 
$\eta(t_k)\in S_{m_k}^\circ$. 
From the transversality, the intersection number 
$\eps_k:=I_{\eta(t_k)}(S_{m_k}^\circ, \eta)$ 
is either $+1$ or $-1$ because of transversality. 
The class $[\eta]\in\pi_1(M(\scA), *)$ is expressed as 
$$
[\eta]=
[\eta_{m_1}]^{\eps_1}
[\eta_{m_2}]^{\eps_2}
\dots
[\eta_{m_N}]^{\eps_N}. 
$$
Thus any $[\eta]\in\pi_1(M)$ is generated by 
$[\eta_1], \dots, [\eta_n]$. 
\qed

\begin{remark}
If we fix the base point in $\scF^1_\bC=\scF^1\otimes\bC$, 
then we may choose transversal generators as in 
Figure \ref{fig:transvgen}. 
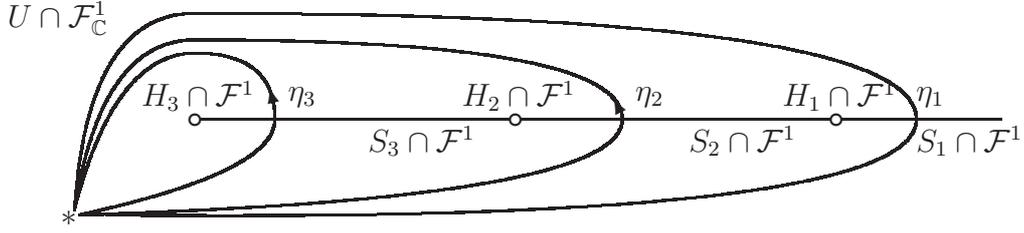
\begin{figure}[htbp]
\begin{picture}(100,100)(20,0)
\thicklines

\put(30,85){$U\cap\scF^1_\bC$}

\put(100,50){\circle{4}}
\put(80,55){$H_3\cap\scF^1$}
\put(220,50){\circle{4}}
\put(200,55){$H_2\cap\scF^1$}
\put(340,50){\circle{4}}
\put(320,55){$H_1\cap\scF^1$}

\put(102,50){\line(1,0){116}}
\put(165,38){$S_3\cap\scF^1$}
\put(222,50){\line(1,0){116}}
\put(285,38){$S_2\cap\scF^1$}
\put(342,50){\line(1,0){60}}
\put(370,38){$S_1\cap\scF^1$}

\put(50,10){$*$}

\qbezier(57,14)(130,30)(130,50)
\qbezier(130,50)(130,75)(100,75)
\qbezier(100,75)(70,75)(55,16)
\thicklines
\put(128,62){\vector(-1,4){0}}
\put(135,57){$\eta_3$}

\qbezier(57,14)(260,20)(260,50)
\qbezier(260,50)(260,80)(100,80)
\qbezier(100,80)(65,80)(55,16)
\thicklines
\put(256.8,58){\vector(-1,2){0}}
\put(265,57){$\eta_2$}

\qbezier(57,14)(370,10)(370,50)
\qbezier(370,50)(370,90)(100,90)
\qbezier(100,90)(60,90)(55,16)
\thicklines
\put(370,57){$\eta_1$}

\end{picture}
     \caption{Transversal generators $\eta_1, \eta_2, \eta_3$.}
\label{fig:transvgen}
\end{figure}

\end{remark}

\subsection{Chamber relations}
\label{subsec:chamber}

As we have seen in the previous section, 
the transversal generators determine a surjective 
homomorphism 
$$
G:
F\langle
\eta_1, \dots, \eta_n
\rangle
\longrightarrow
\pi_1(M(\scA), *), 
$$
from the free group generated by $\eta_1, \dots, \eta_n$ 
to $\pi_1(M(\scA), *)$. 
We will prove that the kernel of the above map is generated 
by conjugacy classes of meridian loops around chambers 
$C\subset M(\scA)$, $C\in\ch_2^\scF(\scA)$. 

Let $\eta:[0,1]\rightarrow M(\scA)$ be a loop with 
$\eta(0)=\eta(1)=*$. Suppose that $\eta$ represents 
an element of $\Ker G$. Then $\eta$ is null-homotopic in 
$M(\scA)$, and hence there is a homotopy 
$\sigma:[0,1]^2\rightarrow M(\scA)$ such that 
$\sigma(t,0)=\eta(t)$, 
$\sigma(t,1)=\sigma(0,s)=\sigma(1,s)=*$. 
We can perturb $\sigma$ in such a way that 
\begin{itemize}
\item $\sigma(\partial[0,1]^2)\cap \bigsqcup_{C\in\ch_2}C=\emptyset$. 
\item $\sigma$ intersects $\bigsqcup_{C\in\ch_2}C$ 
transversely. 
\end{itemize}
Let $\sigma^{-1}(\bigsqcup_{C\in\ch_2}C)=\{q_1, \dots, q_L\}$. 
We choose a meridian loop $v_i$ in $[0,1]^2$ around 
each point $q_i$ with the base point $(0,0)$. 
Let $\alpha:[0,1]\rightarrow \partial([0,1]^2)$ be the 
loop with the base point $(0,0)$ that goes along 
the boundary in the counter clockwise direction. 
Then $\alpha$ is homotopically equivalent to 
a product of meridians $v_1, \dots, v_n$. Since $\eta$ is 
homotopically equivalent to $\sigma\circ\alpha$, 
it is also homotopically equivalent to the 
product of meridian loops $\sigma\circ v_i$ that 
are meridian loops of chambers. (Figure \ref{fig:rel}.)

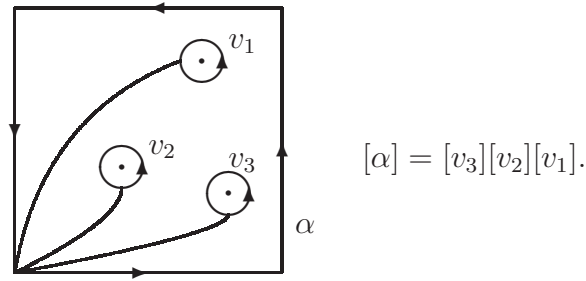
\begin{figure}[htbp]
\begin{picture}(100,100)(20,0)
\thicklines

\multiput(100,0)(0,100){2}{\line(1,0){100}}
\multiput(100,0)(100,0){2}{\line(0,1){100}}

\put(150,0){\vector(1,0){0}}
\put(150,100){\vector(-1,0){0}}
\put(100,50){\vector(0,-1){0}}
\put(200,50){\vector(0,1){0}}
\put(205,15){$\alpha$}

\put(140,40){\circle{16}}
\put(140,40){\circle*{2}}
\put(147.6,43){\vector(0,1){0}}
\qbezier(100,0)(140,20)(140,32)
\put(150,45){$v_2$}

\put(170,80){\circle{16}}
\put(170,80){\circle*{2}}
\put(177.6,83){\vector(0,1){0}}
\qbezier(100,0)(110,60)(162,80)
\put(180,85){$v_1$}

\put(180,30){\circle{16}}
\put(180,30){\circle*{2}}
\put(187.6,33){\vector(0,1){0}}
\qbezier(100,0)(180,15)(180,22)
\put(180,40){$v_3$}

\put(230,40){$[\alpha]=[v_3][v_2][v_1]$.} 

\end{picture}
     \caption{Inverse images of chambers.}\label{fig:rel}
\end{figure}

We will describe the relations more explicitly 
in \S\ref{subsec:dual}.

\subsection{Dual presentation}
\label{subsec:dual}

Let $i=1, \dots, n$ and $C\in\ch_2^\scF(\scA)$. 
We define the $i$-th degree 
$d_i(C)\in\{-1, 0, +1\}$ by 
\begin{equation}
d_i(C)=
\left\{
\begin{array}{cc}
-1& \mbox{ if }\alpha_{i-1}(C)<0<\alpha_i(C), \\
+1& \mbox{ if }\alpha_{i-1}(C)>0>\alpha_i(C), \\
0& \mbox{otherwise}.
\end{array}
\right.
\end{equation}
(Here we use the convention $\alpha_0=-1$, in particular, 
$\alpha_0(C)<0$ for any chamber $C$. See \S\ref{subsec:ex} for 
examples.) 

We will prove (in \S\ref{subsec:proofdual}) 
that the meridian loop of $C\subset M(\scA)$ 
($C\in\ch_2^\scF(\scA)$) is conjugate to the word 
\begin{equation}
\label{eq:releta}
E(C):=
\eta_n^{d_n(C)}
\eta_{n-1}^{d_{n-1}(C)}
\dots
\eta_1^{d_1(C)}
\cdot
\eta_n^{-d_n(C)}
\eta_{n-1}^{-d_{n-1}(C)}
\dots
\eta_1^{-d_1(C)}. 
\end{equation}

\begin{theorem}
\label{thm:dualpres}
With notation as above, the fundamental group 
$\pi_1(M(\scA), *)$ is isomorphic to the group 
defined by the presentation 
$$
\langle
\eta_1, \dots, \eta_n
\mid
E(C), C\in\ch_2^\scF(\scA)
\rangle. 
$$
\end{theorem}

\begin{remark}
The information about homotopy type of 
$M(\scA)$ is encoded in the degree map 
$d_i:\ch_2^\scF(\scA)\rightarrow
\{0,\pm1\}$. 
Indeed, it plays a role when we present 
cellular chain complex with coefficients in 
a local system (see \S\ref{subsec:twisted}). 
\end{remark}


Before proving Theorem \ref{thm:dualpres} we introduce 
some terminology. 

\subsection{Pivotal argument}

Let us denote the argument of the line $H_i$ by $\theta_i$, that 
is the angle of two positive half lines of $\scF^1$ and $H_i$ 
(see Figure \ref{fig:pv}). 
By the assumption on generic flag, arguments 
$\theta_1, \dots, \theta_n$ satisfy 
\begin{equation}
\label{eq:ineq}
0<\theta_n\leq \theta_{n-1}\leq \dots \leq\theta_1<\pi. 
\end{equation}

\begin{remark}
Sometimes it is convenient to define 
$\theta_0:=\theta_1$. 
\end{remark}

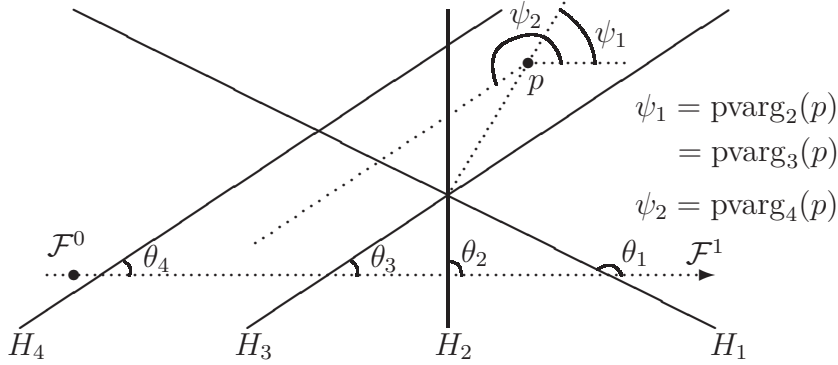
\begin{figure}[htbp]
\begin{picture}(100,120)(20,0)
\thicklines

\put(60,20){\circle*{4}}
\put(50,27){$\scF^0$}

\multiput(50,20)(3,0){83}{\circle*{1}}
\put(300,20){\vector(1,0){0}}
\put(289,24){$\scF^1$}

\put(300,0){\line(-2,1){240}}
\put(298,-10){$H_1$}
\qbezier(256,22)(262,26)(265,20)
\put(265,25){$\theta_1$}

\put(200,0){\line(0,1){120}}
\put(195,-10){$H_2$}
\qbezier(205,20)(205,25)(200,25)
\put(205,25){$\theta_2$}

\put(125,0){\line(3,2){180}}
\put(120,-10){$H_3$}
\qbezier(166,20)(167,23)(164,25)
\put(171,23){$\theta_3$}

\put(40,0){\line(3,2){180}}
\put(35,-10){$H_4$}
\qbezier(81,20)(82,23)(79,25)
\put(86,23){$\theta_4$}

\put(230,100){\circle*{4}}
\put(230,90){$p$}
\multiput(200,50)(1.5,2.5){30}{\circle*{1.25}}

\multiput(230,100)(3,0){13}{\circle*{1.25}}
\qbezier(242,120)(254,112)(255,100)
\put(255,108){$\psi_1$}

\multiput(230,100)(-3,-2){35}{\circle*{1.25}}
\qbezier(242.5,100)(238,116)(222,108)
\qbezier(222,108)(214,104)(218,92)
\put(223,115){$\psi_2$}

\put(270,80){$\psi_1=\pv_2(p)$}
\put(286,63){$=\pv_3(p)$}
\put(270,43){$\psi_2=\pv_4(p)$}

\end{picture}
     \caption{Pivotal arguments.}\label{fig:pv}
\end{figure}

\begin{definition}
Let $p=(x_1, x_2)\in\bR^2$ be a point different from 
$H_i\cap H_{i-1}$. For $i=1, \dots, n$, define the 
{\em $i$-th pivotal argument} $\pv_i(p)\in [0,2\pi)$ by 
$$
\pv_i(p)=
\left\{
\begin{array}{cl}
\arg(\stackrel{\rightarrow}{qp}),&
\mbox{ if $i>1$ and }H_i\cap H_{i-1}(\neq\emptyset)=\{q\},\\
&\\
\theta_i+\pi,&\mbox{ if $i=1$ or $i>1$, $H_i$ is parallel to $H_{i-1}$}. 
\end{array}
\right.
$$
And also 
$$
|\pv_i(p)|=
\left\{
\begin{array}{cl}
\pv_i(p),&
\mbox{ if }0\leq\pv_i(p)<\pi,\\
\pv_i(p)-\pi,&
\mbox{ if }\pi\leq\pv_i(p)<2\pi.
\end{array}
\right.
$$
\end{definition}

We have the following. 

\begin{proposition}
\label{prop:between}
Let $p\in\bR^2$. Suppose $\alpha_i(p)\cdot\alpha_{i-1}(p)<0$ ($i>1$). 
\begin{itemize}
\item 
If $H_{i-1}$ and $H_i$ intersects, 
then $\theta_i<|\pv_i(p)|<\theta_{i-1}$. 
\item 
If $H_{i-1}$ and $H_i$ are parallel, 
then $\theta_i=|\pv_i(p)|=\theta_{i-1}$. 
\end{itemize}
\end{proposition}

Using pivotal arguments, we can describe the intersection 
number of the sail $S_i=S(\alpha_{i-1}, \alpha_i)\cap M(\scA)$ 
and a curve, which is a generalization of Example \ref{ex:int}. 

\begin{example}
\label{ex:int2}
Let $p(x_1, x_2)\in\bR^2\setminus\bigcup_{H\in\scA}H$ and 
$\eps>0$. 
Consider the loop 
$$
\begin{array}{rccl}
\gamma:&\bR/2\pi\bZ&\longrightarrow&M(\scA)\\
&&&\\
&t&\longmapsto&
(x_1, x_2)+
\ii\eps
(\cos t, \sin t). 
\end{array}
$$
If $\alpha_i(p)\cdot\alpha_{i-1}(p)>0$, then 
$\gamma$ does not intersect $S_i$. 
If $\alpha_i(p)\cdot\alpha_{i-1}(p)<0$, then 
$
\gamma^{-1}(S_i)=\{\pv_i(p), \pv_i(p)+\pi \}. 
$
We have 
\begin{eqnarray*}
&&I_{\gamma(\pv_i(p))}(S_i, \gamma)=1, \mbox{ and} \\
&&I_{\gamma(\pv_i(p)+\pi)}(S_i, \gamma)=-1.
\end{eqnarray*}
\end{example}

Combining this with the degree $d_i$, we have the 
following. 

\begin{proposition}
\label{prop:int3}
Let $p(x_1, x_2)\in\bR^2\setminus\bigcup_{H\in\scA}H$ and 
the loop $\gamma$ be as in Example \ref{ex:int2}. 
Let us denote by $C$ the chamber which contains $p$. 
We have 
\begin{eqnarray*}
&&I_{\gamma(|\pv_i(p)|)}(S_i, \gamma)=d_i(C), \mbox{ and} \\
&&I_{\gamma(|\pv_i(p)|+\pi)}(S_i, \gamma)=-d_i(C).
\end{eqnarray*}
\end{proposition}

\subsection{Proof of Theorem \ref{thm:dualpres}}
\label{subsec:proofdual}

Now we prove Theorem \ref{thm:dualpres}. 
Let $C\in\ch_2^\scF(\scA)$ and $p\in C$. 
We take a meridian loop $\gamma:\bR/2\pi\bZ\rightarrow
M(\scA), t\mapsto \gamma(t)$ as in Example \ref{ex:int2}. 
Then $\gamma$ intersects $S_i$ at $t=|\pv_i(p)|$ and 
$t=|\pv_i(p)|+\pi$ with intersection numbers 
$d_i(C)$ and $-d_i(C)$, respectively. (This logically 
includes that $\gamma$ does not intersect $C$ 
if and only if $d_i(C)=0$.) 
In particular, from Proposition \ref{prop:between}, 
$\theta_{i-1}\leq|\pv_i(p)|\leq\theta_i$ provided 
$d_i(C)\neq 0$. 
From Eq. (\ref{eq:ineq}), 
the loop $\gamma$ intersects 
$S_n, S_{n-1}, \dots, S_1, S_n, S_{n-1}, \dots, S_1$ 
in this order 
with intersection numbers 
$d_n(C), d_{n-1}(C), \dots, d_1(C), 
-d_n(C), -d_{n-1}(C), \dots, -d_1(C)$. 
Hence the loop $\gamma$ is homotopic to the word 
$E(C)$ in Eq. (\ref{eq:releta}).

\subsection{Examples}
\label{subsec:ex}

\begin{example}
\label{ex:5lines}
Let $\scA=\{H_1, \dots, H_5\}$ be 
a line arrangement and $\scF$ be a 
flag pictured in Figure \ref{fig:5lines}. Then 
$\ch_2^\scF(\scA)=\{C_6, C_7, \dots, C_{12}\}$ 
consists of $7$ chambers. The degrees 
can be computed as follows. 
$$
\begin{array}{c|ccccc}
       &d_1&d_2&d_3&d_4&d_5\\
\hline
C_6    & 0 &0  &-1 & 1 &-1 \\
C_7    & 0 &-1 & 0 & 1 &-1 \\
C_8    & 0 &-1 & 1 & 0 &-1 \\
C_9    & 0 &-1 & 1 & 0 & 0 \\
C_{10} &-1 &0  & 1 & 0 & 0 \\
C_{11} &-1 &0  & 0 & 1 & 0 \\
C_{12} &-1 &0  & 0 & 1 &-1 
\end{array}
$$
The fundamental group $\pi_1(M(\scA), *)$ has the following 
presentation. 
\begin{equation*}
\begin{split}
\pi_1(M(\scA),*)=
\langle
\eta_1, \dots, \eta_5\mid 
E(C_6):&\  
\eta_5^{-1}\eta_4^{}\eta_3^{-1}
\eta_5^{}\eta_4^{-1}\eta_3^{}\\
E(C_7):&\  
\eta_5^{-1}\eta_4^{}\eta_2^{-1}
\eta_5^{}\eta_4^{-1}\eta_2^{}\\
E(C_8):&\  
\eta_5^{-1}\eta_3^{}\eta_2^{-1}
\eta_5^{}\eta_3^{-1}\eta_2^{}\\
E(C_9):&\  
\eta_3^{}\eta_2^{-1}
\eta_3^{-1}\eta_2^{}\\
E(C_{10}):&\  
\eta_3^{}\eta_1^{-1}
\eta_3^{-1}\eta_1^{}\\
E(C_{11}):&\  
\eta_4^{}\eta_1^{-1}
\eta_4^{-1}\eta_1^{}\\
E(C_{12}):&\  
\eta_5^{-1}\eta_4^{}\eta_1^{-1}
\eta_5^{}\eta_4^{-1}\eta_1^{}\rangle. 
\end{split}
\end{equation*}

\begin{figure}[htbp]
\begin{picture}(100,110)(20,-10)
\thicklines

\put(60,10){\circle*{4}}
\put(50,17){$\scF^0$}

\multiput(50,10)(3,0){108}{\circle*{1}}
\put(350,10){\vector(1,0){0}}
\put(339,14){$\scF^1$}

\put(350,0){\line(-2,1){200}}
\put(350,-10){$H_1$}

\put(280,0){\line(-2,1){200}}
\put(280,-10){$H_2$}

\put(150,0){\line(2,3){67}}
\put(150,-10){$H_3$}

\put(120,0){\line(2,1){200}}
\put(110,-10){$H_4$}

\put(50,0){\line(2,1){200}}
\put(40,-10){$H_5$}

\footnotesize
\put(100,50){$C_0$}
\put(300,50){$C_1$}
\put(240,30){$C_2$}
\put(195,20){$C_3$}
\put(130,-5){$C_4$}
\put(130,20){$C_5$}
\put(178,37){$C_6$}
\put(200,50){$C_7$}
\put(172,54){$C_8$}
\put(150,80){$C_9$}
\put(190,90){$C_{10}$}
\put(216,93){$C_{11}$}
\put(250,80){$C_{12}$}

\end{picture}
\caption{Example \ref{ex:5lines} and \ref{ex:5lines2}.}
\label{fig:5lines}
\end{figure}
\end{example}

\begin{example}
\label{ex:6lines}
Let $\scA=\{H_1, \dots, H_5, H_6\}$ be 
a line arrangement and $\scF$ be a 
flag pictured in Figure \ref{fig:6lines}. Then 
$\ch_2^\scF(\scA)=\{C_7, C_8, \dots, C_{17}\}$ 
consists of $11$ chambers. The degrees 
can be computed as follows. 
$$
\begin{array}{c|cccccc}
       &d_1&d_2&d_3&d_4&d_5&d_6\\
\hline
C_7    &-1 & 1 & 0 &-1 & 0 &0 \\
C_8    &-1 & 1 & 0 & 0 &-1 &0 \\
C_9    &-1 & 1 & 0 & 0 &-1 &1 \\
C_{10} &-1 &0  & 1 & 0 &-1 &1 \\
C_{11} &-1 &0  & 1 &-1 & 0 &1 \\
C_{12} &-1 &0  & 1 &-1 & 0 &0 \\
C_{13} &-1 & 1 & 0 & 0 & 0 &0 \\
C_{14} &-1 &0  & 1 & 0 & 0 &0 \\
C_{15} &-1 &0  & 0 & 1 & 0 &0 \\
C_{16} &-1 &0  & 0 & 0 & 1 &0 \\
C_{17} &-1 &0  & 0 & 0 & 0 &1 
\end{array}
$$
The fundamental group $\pi_1(M(\scA), *)$ has the following 
presentation. 
\begin{equation*}
\begin{split}
\pi_1(M(\scA),*)=
\langle
\eta_1, \dots, \eta_6\mid 
E(C_7):&\ 
\eta_4^{-1}\eta_2^{}\eta_1^{-1}
\eta_4^{}\eta_2^{-1}\eta_1^{}\\
E(C_8):&\ 
\eta_5^{-1}\eta_2^{}\eta_1^{-1}
\eta_5^{}\eta_2^{-1}\eta_1^{}\\
E(C_9):&\ 
\eta_6^{}\eta_5^{-1}\eta_2^{}\eta_1^{-1}
\eta_6^{-1}\eta_5^{}\eta_2^{-1}\eta_1^{}\\
E(C_{10}):&\ 
\eta_6^{}\eta_5^{-1}\eta_3^{}\eta_1^{-1}
\eta_6^{-1}\eta_5^{}\eta_3^{-1}\eta_1^{}\\
E(C_{11}):&\ 
\eta_6^{}\eta_4^{-1}\eta_3^{}\eta_1^{-1}
\eta_6^{-1}\eta_4^{}\eta_3^{-1}\eta_1^{}\\
E(C_{12}):&\ 
\eta_4^{-1}\eta_3^{}\eta_1^{-1}
\eta_4^{}\eta_3^{-1}\eta_1^{}\\
E(C_{13}):&\ 
\eta_2^{}\eta_1^{-1}
\eta_2^{-1}\eta_1^{}\\
E(C_{14}):&\ 
\eta_3^{}\eta_1^{-1}
\eta_3^{-1}\eta_1^{}\\
E(C_{15}):&\ 
\eta_4^{}\eta_1^{-1}
\eta_4^{-1}\eta_1^{}\\
E(C_{16}):&\ 
\eta_5^{}\eta_1^{-1}
\eta_5^{-1}\eta_1^{}\\
E(C_{17}):&\ 
\eta_6^{}\eta_1^{-1}
\eta_6^{-1}\eta_1^{}\rangle
\end{split}
\end{equation*}
The relations 
$E(C_{13}), \dots, E(C_{17})$, indicate that 
the large loop $\eta_1$ is contained in 
the center of the group. 

\begin{figure}[htbp]
\begin{picture}(100,120)(20,0)
\thicklines

\put(40,20){\circle*{4}}
\put(30,26){$\scF^0$}

\multiput(35,20)(3,-0.05){114}{\circle*{1}}
\put(360,15){\vector(1,0){0}}
\put(349,19){$\scF^1$}

\qbezier(360,10)(120,35)(60,40)
\put(360,0){$H_1$}

\qbezier(259,10)(220,60)(190,98)
\put(265,0){$H_2$}

\put(250,10){\line(0,1){110}}
\put(233,3){$H_3$}

\put(182.5,10){\line(3,4){82.5}}
\put(172,0){$H_4$}

\put(70,10){\line(2,1){220}}
\put(70,0){$H_5$}

\put(40,15){\line(4,1){330}}
\put(30,5){$H_6$}

\footnotesize 
\put(60,29){$C_0$}
\put(300,40){$C_1$}
\put(300,0){$C_2$}
\put(250,0){$C_3$}
\put(210,0){$C_4$}
\put(130,0){$C_5$}
\put(53,8){$C_6$}
\put(215,30){$C_7$}
\put(180,35){$C_8$}
\put(195,60){$C_9$}
\put(214,74){$C_{10}$}
\put(234,72){$C_{11}$}
\put(234,50){$C_{12}$}
\put(120,60){$C_{13}$}
\put(210,100){$C_{14}$}
\put(250,120){$C_{15}$}
\put(271,120){$C_{16}$}
\put(271,90){$C_{17}$}

\end{picture}
\caption{Example \ref{ex:6lines} and \ref{ex:6lines2}.}
\label{fig:6lines}
\end{figure}
\end{example}

\begin{remark}
Example \ref{ex:6lines} 
gives a presentation for the pure braid group 
with $4$-strands. See also Example \ref{ex:6lines2}. 
\end{remark}

\subsection{Twisted minimal chain complex}
\label{subsec:twisted}

Let $\scL$ be a complex rank one local system 
on $M(\scA)$. $\scL$ is determined by 
nonzero complex numbers (monodromy around $H_i$) 
$q_i\in\bC^*$, $i=1, \dots, n$. 
Fix a square root $q_i^{1/2}\in\bC^*$ for each $i$. 
For given chambers $C, C'$, let us define 
$$
\Delta(C,C'):=
\prod_{H_i\in\Sep(C,C')}q^{1/2}
-
\prod_{H_i\in\Sep(C,C')}q^{-1/2}, 
$$
where $H_i\in\Sep(C,C')$ runs over all 
hyperplanes which separate $C$ and $C'$. 
With these notation, we can describe 
a chain complex which computes homology 
groups with coefficients in $\scL$. 

\begin{theorem}
\label{thm:twist}
Denote by 
$\bC[\ch_k^\scF(\scA)]:=\bigoplus_{C\in\ch_k^\scF(\scA)}\bC\cdot[C]$ 
the vector space spanned by $\ch_k^\scF(\scA)$. 
Recall that $\ch_1^\scF(\scA)=\{C_1, C_2, \dots, C_n\}$ and 
$\ch_0^\scF(\scA)=\{C_0\}$. 
Then the linear maps 
\begin{equation*}
\begin{split}
\nabla:\ch_2^\scF(\scA)\longrightarrow\ch_1^\scF(\scA), &\ 
[C]\longmapsto
\sum_{i=1}^n
d_i(C)\Delta(C,C_i)[C_i], \\
\nabla:\ch_1^\scF(\scA)\longrightarrow
\ch_0^\scF(\scA), &\ 
[C_i]\longmapsto\Delta(C_0,C_i)[C_0], 
\end{split}
\end{equation*}
determines a chain complex $(\bC[\ch_\bullet^\scF(\scA)], \nabla)$ 
which homology group is isomorphic to 
$$
H_k(\bC[\ch_\bullet^\scF(\scA)], \nabla)
\simeq
H_k(M(\scA), \scL). 
$$
\end{theorem}
See \cite{yos-ch, yos-loc} for details and applications.

\section{Positive homogeneous presentations}
\label{sec:poshom}

\subsection{Left and right lines}

In this section, we give an alternative 
presentation for the fundamental group $\pi_1(M(\scA))$. 
It is presented with positive homogeneous relations as: 
\begin{equation*}
\begin{split}
\mbox{Generators}:&\ \gamma_1, \gamma_2, \dots, \gamma_n, \\
\mbox{Relations}, R(C):&\ \gamma_1\gamma_2\dots\gamma_n
=\gamma_{i_1(C)}\gamma_{i_2(C)}\dots\gamma_{i_n(C)}, 
\end{split}
\end{equation*}
where $C$ runs over all $\ch_2^\scF(\scA)$ and 
$(i_1(C), \dots, i_n(C))$ is a permutation 
of $(1, \dots, n)$ associated to 
$C$. 

\begin{definition}
Let $C\in\ch(\scA)$ be a chamber. 
The line $H_i\in\scA$ is said to be {\em passing 
the left side of $C$} if $C\subset\{\alpha_i>0\}$. 
Similarly, 
The line $H_i\in\scA$ is said to be {\em passing 
the right side of $C$} if $C\subset\{\alpha_i<0\}$. 
\end{definition}

\begin{remark}
Sometimes it is convenient to consider 
$0$-th line $H_0$ is passing the right side 
of $C$ for any chamber $C$. (Recall that 
$\alpha_0(C)=-1$ by our convention.) 
\end{remark}

\begin{definition}
For a chamber $C\in\ch(\scA)$, define the 
decomposition 
$\{1, \dots, n\}=I_R(C)\sqcup I_L(C)$ as follows. 
\begin{equation*}
\begin{split}
I_R(C)=&\{i\mid H_i\mbox{ passes the right side of }C\},\\
I_L(C)=&\{i\mid H_i\mbox{ passes the left side of }C\}. 
\end{split}
\end{equation*}
\end{definition}
The notion right/left is related to 
the map $d_i$. The proof of the next proposition 
is straightforward. 
\begin{proposition}
\label{prop:straightf}
Let $C\in\ch_2^\scF(\scA)$. 
\begin{itemize}
\item 
If 
$H_{i-1}$ is passing right side of $C$ and 
$H_{i}$ is passing left side of $C$, then 
$d_i(C)=-1$. 
\item 
If 
$H_{i-1}$ is passing left side of $C$ and 
$H_{i}$ is passing right side of $C$, then 
$d_i(C)=1$. 
\item 
Otherwise, $d_i(C)=0$. 
\end{itemize}
\end{proposition}

\subsection{Positive homogeneous relations}

For a chamber $C\in\ch_2^\scF(\scA)$, arranging 
the right/left indices increasingly as 
\begin{equation*}
\begin{split}
I_R(C)=&\{i_1(C)<i_2(C)<\dots< i_k(C)\},\\
I_L(C)=&\{i_{k+1}(C)<i_{k+2}(C)<\dots< i_n(C)\}. 
\end{split}
\end{equation*}
Then we introduce the following homogeneous relation. 
\begin{equation}
\label{eq:homog}
\Gamma(C): 
\gamma_1\gamma_2\dots\gamma_n=
\gamma_{i_1(C)}\gamma_{i_2(C)}\dots\gamma_{i_n(C)}. 
\end{equation}

\begin{theorem}
\label{thm:poshom}
With notation as above, the fundamental group 
$\pi_1(M(\scA), *)$ is isomorphic to the group 
defined by the presentation 
$$
\langle
\gamma_1, \dots, \gamma_n
\mid
\Gamma(C), C\in\ch_2^\scF(\scA)
\rangle. 
$$
\end{theorem}

\begin{remark}
Note that all relations in the above presentation 
are positive homogeneous. It is similar to the 
``conjugation-free geometric presentation'' introduced 
in \cite{egt1, egt2}. However they require stronger 
properties on relations. Indeed they prove that 
the fundamental group of 
Ceva arrangement (Figure \ref{fig:6lines}) 
does not have conjugation-free geometric presentation. 
\end{remark}

\begin{example}
\label{ex:5lines2}
Let $\scA=\{H_1, \dots, H_5\}$ be 
a line arrangement and $\scF$ be a 
flag pictured in Figure \ref{fig:5lines}. 
$$
\begin{array}{c|c|c}
&I_R(C)&I_L(C)\\
\hline
C_6&124&35\\
C_7&14&235\\
C_8&134&25\\
C_9&1345&2\\
C_{10}&345&12\\
C_{11}&45&123\\
C_{12}&4&1235
\end{array}
$$
Hence the fundamental group has the following 
presentation. 
\begin{equation*}
\begin{split}
\pi_1(M(\scA),*)\simeq
\langle\gamma_1, \dots, \gamma_5\mid&
12345=12435=14235=13425\\
&=13452=34512=45123=41235\rangle. 
\end{split}
\end{equation*}
Here we denote $12345$ instead of 
$\gamma_1\gamma_2\gamma_3\gamma_4\gamma_5$ for 
simplicity. 
\end{example}

\begin{example}
\label{ex:6lines2}
Let $\scA=\{H_1, \dots, H_6\}$ be 
a line arrangement and $\scF$ be a 
flag pictured in Figure \ref{fig:6lines}. 
$$
\begin{array}{c|c|c}
&I_R(C)&I_L(C)\\
\hline
C_7&23&1456\\
C_8&234&156\\
C_9&2346&15\\
C_{10}&346&125\\
C_{11}&36&1245\\
C_{12}&3&12456\\
C_{13}&23456&1\\
C_{14}&3456&12\\
C_{15}&456&123\\
C_{16}&56&1234\\
C_{17}&6&12345
\end{array}
$$
Hence the fundamental group has the following 
presentation. 
\begin{equation*}
\begin{split}
\pi_1(M(\scA),*)\simeq
\langle\gamma_1, \dots, \gamma_6\mid&
123456\\
&=231456=234156=234615=346125=361245=312456\\
&=234561=345612=456123=561234=612345
\rangle. 
\end{split}
\end{equation*}
\end{example}

\subsection{Proof of Theorem \ref{thm:poshom}}

The new presentation in Theorem \ref{thm:poshom} is 
obtained by changing generators as 
$\eta_i=\gamma_i\gamma_{i+1}\dots\gamma_n$, or equivalently, 
\begin{equation}
\label{eq:change}
\begin{split}
\gamma_1=&\eta_1\eta_2^{-1}\\
\gamma_2=&\eta_2\eta_3^{-1}\\
\dots&\\
\gamma_{n-1}=&\eta_{n-1}\eta_n^{-1}\\
\gamma_n=&\eta_n. 
\end{split}
\end{equation}

\begin{remark}
If we fix the base point in $\scF^1_\bC=\scF^1\otimes\bC$, 
then we may choose meridian generators 
$\gamma_1, \dots, \gamma_n$ as in 
Figure \ref{fig:gamma}. 
(Compare Figure \ref{fig:transvgen}.)
\begin{figure}[htbp]
\begin{picture}(100,100)(20,0)
\thicklines

\put(30,85){$U\cap\scF^1_\bC$}

\put(100,50){\circle{4}}
\put(80,55){$H_3\cap\scF^1$}
\put(220,50){\circle{4}}
\put(200,55){$H_2\cap\scF^1$}
\put(340,50){\circle{4}}
\put(320,55){$H_1\cap\scF^1$}

\put(102,50){\line(1,0){116}}
\put(222,50){\line(1,0){116}}
\put(342,50){\line(1,0){60}}

\put(50,10){$*$}

\qbezier(57,14)(130,30)(130,50)
\qbezier(130,50)(130,75)(100,75)
\qbezier(100,75)(70,75)(55,16)
\thicklines
\put(128,62){\vector(-1,4){0}}
\put(135,57){$\gamma_3$}

\qbezier(57,14)(260,20)(260,50)
\qbezier(260,50)(260,80)(190,80)
\qbezier(190,80)(170,80)(170,50)
\qbezier(170,50)(170,21)(160,19.5)
\thicklines
\put(257.99,59){\vector(-1,3){0}}
\put(265.5,58.5){$\gamma_2$}

\qbezier(57,14)(370,10)(370,50)
\qbezier(370,50)(370,90)(320,90)
\qbezier(320,90)(300,90)(300,50)
\qbezier(300,50)(300,22)(290,21)
\thicklines
\put(373,57){$\gamma_1$}
\put(368.5,62){\vector(-1,4){0}}

\end{picture}
     \caption{Meridian generators $\gamma_1, \gamma_2, \gamma_3$.}
\label{fig:gamma}
\end{figure}
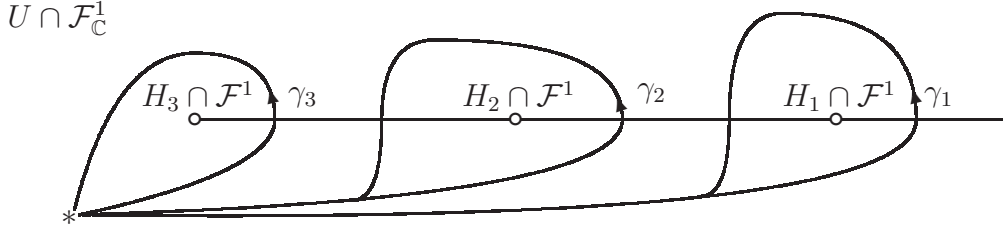

\end{remark}

\begin{proposition}
\label{prop:equiv}
By the change (\ref{eq:change}), 
the relation $E(C)=1$ (Eq. (\ref{eq:releta})) 
is equivalent to $\Gamma(C)$ (Eq. (\ref{eq:homog})). 
\end{proposition}

\begin{proof}
We distinguish four cases according to 
$H_1$ and $H_n$ are passing right/left of $C$. 

Case (1). 
Both $H_1$ and $H_n$ are passing right 
side of $C$. 

Case (2). 
$H_1$ is passing right and $H_n$ is passing left 
side of $C$. 

Case (3). 
Both $H_1$ and $H_n$ are passing left 
side of $C$. 

Case (4). 
$H_1$ is passing left and $H_n$ is passing right 
side of $C$. 

\noindent
Case (1). We may take 
$1<i_1<\dots<i_{2k}<n$ in such a way that 
$$
\stackrel{\mbox{\footnotesize right}}
{\overbrace{1, 2, \dots}}\ , 
\stackrel{\mbox{\footnotesize left}}
{\overbrace{i_1, i_1+1, \dots}}\  , 
\stackrel{\mbox{\footnotesize right}}
{\overbrace{i_2, i_2+1, \dots}}\  , 
\dots\  , 
\stackrel{\mbox{\footnotesize right}}
{\overbrace{i_{2k}, i_{2k}+1, \dots, n}}. 
$$
In this case we have 
\begin{equation*}
\begin{split}
I_R(C)=&
\{1, 2, \dots, i_1-1, i_2, i_2+1,\dots, i_3-1, \dots, 
i_{2k}, i_{2k}+1\dots, n\},\\
I_L(C)=&
\{i_1, i_1+1, \dots, i_2-1, i_3, \dots, i_4-1, \dots, 
i_{2k-1}, i_{2k-1}+1, \dots, i_{2k}-1\}.
\end{split}
\end{equation*}
Then by Proposition \ref{prop:straightf}, 
$
d_{i_{2g-1}}(C)=-1, 
d_{i_{2g}}(C)=1$ ($g=1, \dots, k$) 
and otherwise, 
$d_{i}(C)=0$. Hence the word $E(C)$ is equal to 
$$
E(C)=
\eta_{i_1}^{-1}
\eta_{i_2}^{1}
\eta_{i_3}^{-1}
\dots
\eta_{i_{2k}}^{1}
\cdot
\eta_{i_1}^{1}
\eta_{i_2}^{-1}
\eta_{i_3}^{1}
\dots
\eta_{i_{2k}}^{-1}. 
$$
Using (\ref{eq:change}), we have 
\begin{equation*}
\begin{split}
E(C)=&
\eta_{i_1}^{-1}
(\eta_{i_2}\eta_{i_3}^{-1})
\dots
(\eta_{i_{2k-2}}\eta_{i_{2k-1}}^{-1})
\eta_{i_{2k}}
\cdot
(\eta_{i_1}\eta_{i_2}^{-1})
\dots
(\eta_{i_{2k-1}}\eta_{i_{2k}}^{-1})
\\
=&
\eta_{i_1}^{-1}
\cdot
(\gamma_{i_2}\dots\gamma_{i_3-1})\dots
(\gamma_{i_{2k-2}}\dots\gamma_{i_{2k-1}-1})
\cdot
(\gamma_{i_{2k}}\dots\gamma_n)\\
&
\cdot
(\gamma_{i_1}\dots\gamma_{i_2-1})\dots
(\gamma_{i_{2k-1}}\dots\gamma_{i_{2k}-1}). 
\end{split}
\end{equation*}
Since the equality $E(C)=e$ holds, by multiplying 
$\gamma_1\gamma_2\dots\gamma_n$ from the left, we have 
(note that $\gamma_1\gamma_2\dots\gamma_n\eta_{i_1}^{-1}
=\gamma_1\gamma_2\dots\gamma_{i_1-1}$) 
\begin{equation*}
\begin{split}
\gamma_1\gamma_2\dots\gamma_n=&
(\gamma_1\gamma_2\dots\gamma_{i_1-1})
(\gamma_{i_2}\dots\gamma_{i_3-1})
\dots
(\gamma_{i_{2k}}\dots\gamma_n)\\
&\cdot
(\gamma_{i_1}\gamma_{i_1+1}\dots\gamma_{i_2-1})
(\gamma_{i_3}\dots\gamma_{i_4-1})
\dots
(\gamma_{i_{2k-1}}\dots\gamma_{i_{2k}-1}), 
\end{split}
\end{equation*}
which is identical to the relation 
$\Gamma(C)$. 

The remaining cases (2), (3) and (4) are 
handled in the same way. 
\end{proof}

\section{Proofs of main results}
\label{sec:mainproof}

In this section, we prove 
Theorem \ref{thm:main}. 
For this purposes, it is convenient 
to describe $M(\scA)$ in terms of 
tangent bundle of $\bR^2$. 

\subsection{Tangent bundle description}

We identify $\bC^2$ with the total space 
$T\bR^2$ of the tangent bundle of $\bR^2$ 
via 
\begin{equation*}
\begin{split}
T\bR^2\longrightarrow&\bC^2\\
(\bm{x}, \bm{y})\longmapsto&
\bm{x}+\ii\bm{y}, 
\end{split}
\end{equation*}
where $\bm{y}\in T_{\bm{x}}\bR^2$ is a tangent 
vector of $\bR^2$ at $\bm{x}\in\bR^2$. 
Let $H\subset\bR^2$ be a line and $H_\bC\subset\bC^2$ be 
its complexification. 
Then $H_\bC$ is identified by the above map 
with 
\begin{equation}
\label{eq:tangentline}
H_\bC\simeq
\{(\bm{y}\in T_{\bm{x}}\bR^2)\mid 
\bm{x}\in H, \bm{y}\in T_{\bm{x}}H\}. 
\end{equation}
For $\bm{x}\in\bR^2$, write $\scA_{\bm{x}}$ 
the set of lines passing through $\bm{x}$. 
Then we have the following 
(see \cite[\S3.1]{yos-lef}.): 
$$
M(\scA)\simeq
\{(\bm{y}\in T_{\bm{x}}\bR^2)\mid 
\bm{x}\in\bR^2, 
\bm{y}\notin T_{\bm{x}}H, 
\mbox{ for }H\in\scA_{\bm{x}}\}. 
$$
It is straightforward to check the 
following from 
(\ref{eq:tangentline}). 
\begin{lemma}
\label{lem:unifmot}
If $\bm{x}+\ii\bm{y}\in M(\scA)$, then 
$(\bm{x}+t\bm{y})+\ii\bm{y}\in M(\scA)$ 
for any $t\in\bR$. 
\end{lemma}
Thus lines and the complement $M(\scA)$ are 
preserved under the linear uniform motion. 
The next lemma shows that the sail $S(\alpha, \beta)$ 
is also preserved under the linear uniform motion. 
The next lemma will be used repeatedly to construct 
deformation retractions for certain subsets of $M(\scA)$. 

\begin{lemma}
\label{lem:unifmotsail}
Let $\alpha, \beta$ be linear forms 
(as in Definition \ref{def:sail}). 
Suppose $\bm{x}+\ii\bm{y}\in S(\alpha, \beta)$. 
Then $(\bm{x}+t\bm{y})+\ii\bm{y}\in S(\alpha, \beta)$ 
for any $t\in\bR$. 
Conversely, if $\bm{x}+\ii\bm{y}\notin S(\alpha, \beta)$, 
then $(\bm{x}+t\bm{y})+\ii\bm{y}\notin S(\alpha, \beta)$ 
for any $t\in\bR$. 
\end{lemma}

\begin{proof}
Set 
$\alpha(\bm{x})=\bm{a}\cdot\bm{x}+b$ and 
$\beta(\bm{x})=\bm{c}\cdot\bm{x}+d$, 
where $\bm{a, c}\in(\bR^2)^*$ and $b, d\in\bR$. 
By assumption, 
$$
\frac{\alpha(\bm{x}+\ii\bm{y})}{\beta(\bm{x}+\ii\bm{y})}
=
\frac{\bm{a}\cdot\bm{x}+\ii\bm{a}\cdot\bm{y}+b}{\bm{c}\cdot\bm{x}+\ii\bm{c}\cdot\bm{y}+d}
=
\frac{\alpha(\bm{x})+\ii\bm{a}\cdot\bm{y}}{\beta(\bm{x})+\ii\bm{c}\cdot\bm{y}}
=r\in\bR_{<0}. 
$$
Hence
\begin{equation}
\label{eq:descriptsail}
\alpha(\bm{x})=r\beta(\bm{x})
\mbox{ and }
\bm{a}\cdot\bm{y}=r\bm{c}\cdot\bm{y}.
\end{equation}
The assertion follows from 
\begin{equation}
\label{eq:motion}
\frac{\alpha(\bm{x}+t\bm{y}+\ii\bm{y})}{\beta(\bm{x}+t\bm{y}+\ii\bm{y})}
=
\frac{\alpha(\bm{x})+t\bm{a}\cdot\bm{y}+\ii\bm{a}\cdot\bm{y}}{\beta(\bm{x})+t\bm{c}\cdot\bm{y}+\ii\bm{c}\cdot\bm{y}}
=r.
\end{equation}
The second part follows immediately from the 
first part. 
\end{proof}
Suppose that $\bm{x}+\ii\bm{y}\in S(\alpha, \beta)$ 
and $\bm{a}\cdot\bm{y}\neq 0$. 
Set $t=-\frac{\alpha(\bm{x})}{\bm{a}\cdot\bm{y}}$. 
Then by (\ref{eq:motion}) above, 
$\alpha(\bm{x})+t\bm{a}\cdot\bm{y}=\alpha(\bm{x}+t\bm{y})=0$ and 
$\beta(\bm{x})+t\bm{c}\cdot\bm{y}=\beta(\bm{x}+t\bm{y})=0$, 
which implies that the line $\bm{x}+\bR\cdot\bm{y}$ 
is passing through the intersection $H_\alpha\cap H_\beta$ 
of two lines $H_\alpha=\{\alpha=0\}$ and 
$H_\beta=\{\beta=0\}$. 
We obtain the following description of the sail. 

\begin{proposition}
\label{prop:sails}
Let $\alpha$ and $\beta$ be as in Lemma \ref{lem:unifmotsail}. \\
\begin{itemize}
\item[(i)] Suppose $H_\alpha=\{\alpha=0\}$ and $H_\beta=\{\beta=0\}$ are 
not parallel. 
Then $\bm{x}+\ii\bm{y}\in S(\alpha, \beta)$ if 
and only if either 
\begin{itemize}
\item $\alpha(\bm{x})\beta(\bm{x})<0$ and 
$\bm{y}$ is tangent to the line 
$\overline{\bm{x}\cdot(H_\alpha\cap H_\beta)}$ passing 
through $\bm{x}$ and the intersection $H_\alpha\cap H_\beta$, 
or
\item $\alpha(\bm{x})=\beta(\bm{x})=0$ (i.e., $\{\bm{x}\}=
H_\alpha\cap H_\beta$) and $\bm{y}\neq 0$ such that 
the line $\bm{x}+\bR\cdot\bm{y}$ is passing through 
the domain $\{\bm{x}\in\bR^2\mid \alpha(\bm{x})\beta(\bm{x})<0\}$. 
\end{itemize}
\item[(ii)] Suppose $H_\alpha$ and $H_\beta$ are parallel. 
Then $\bm{x}+\ii\bm{y}\in S(\alpha, \beta)$ if 
and only if $\alpha(\bm{x})\beta(\bm{x})<0$ and 
$\bm{y}$ is either zero or parallel to $H_\alpha$. \\
\item[(iii)] Suppose $\alpha$ is a nonzero constant. (In this 
case, $\beta$ should be degree one.) 
Then $\bm{x}+\ii\bm{y}\in S(\alpha, \beta)$ if 
and only if $\alpha(\bm{x})\beta(\bm{x})<0$ and 
$\bm{y}$ is either zero or parallel to $H_\beta$. 
\end{itemize}
(See Figure \ref{fig:descriptsails}.)
\end{proposition}

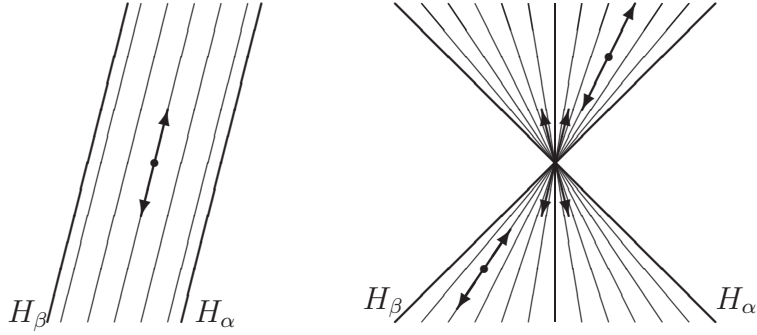
\begin{figure}[htbp]
\begin{picture}(100,120)(20,0)
\thicklines

\put(120,0){\line(1,4){30}}
\put(125,0){$H_\alpha$}

\put(70,0){\line(1,4){30}}
\put(55,0){$H_\beta$}

\thinlines
\multiput(75,0)(10,0){5}{\line(1,4){30}}

\thicklines
\put(110,60){\circle*{3}}
\put(110,60){\vector(1,4){5}}
\put(110,60){\vector(-1,-4){5}}

\put(188,5){$H_\beta$}
\put(200,0){\line(1,1){120}}

\put(320,0){\line(-1,1){120}}
\put(320,5){$H_\alpha$}

\thinlines
\put(210,0){\line(5,6){100}}
\put(220,0){\line(2,3){80}}
\put(230,0){\line(1,2){60}}
\put(240,0){\line(1,3){40}}
\put(250,0){\line(1,6){20}}
\put(260,0){\line(0,1){120}}
\put(270,0){\line(-1,6){20}}
\put(280,0){\line(-1,3){40}}
\put(290,0){\line(-1,2){60}}
\put(300,0){\line(-2,3){80}}
\put(310,0){\line(-5,6){100}}

\thicklines
\put(260,60){\vector(1,4){5}}
\put(260,60){\vector(-1,4){5}}
\put(260,60){\vector(1,-4){5}}
\put(260,60){\vector(-1,-4){5}}

\put(233.333,20){\circle*{3}}
\put(233.333,20){\vector(2,3){10}}
\put(233.333,20){\vector(-2,-3){10}}

\put(280,100){\circle*{3}}
\put(280,100){\vector(1,2){10}}
\put(280,100){\vector(-1,-2){10}}

\end{picture}
     \caption{Sails $S(\alpha, \beta)$.}
\label{fig:descriptsails}
\end{figure}

Define 
$$
|\arg(\bm{y})|:=
\left\{
\begin{array}{cl}
\arg(\bm{y}),&
\mbox{ if }0\leq\arg(\bm{y})<\pi\\
\arg(\bm{y})-\pi,&
\mbox{ if }\pi\leq\arg(\bm{y})<2\pi. 
\end{array}
\right.
$$

Using the above and Proposition \ref{prop:between}, 
we have 
\begin{proposition}
\label{prop:sailsM}
Let $\bm{x}+\ii\bm{y}\in S_i=S(\alpha_{i-1}, \alpha_i)\cap M(\scA)$. 
\begin{itemize}
\item 
If $\bm{x}$ is the intersection $H_{i-1}\cap H_i$, then 
$\bm{y}\neq\bm{0}$ and $\theta_{i-1}<|\arg(\bm{y})|<\theta_i$. 
\item 
If $\bm{x}$ is not the intersection $H_{i-1}\cap H_i$, then 
$\alpha_{i-1}(\bm{x})\alpha_{i}(\bm{x})<0$ and $\bm{y}\neq\bm{0}$ with 
$|\arg(\bm{y})|=|\pv_i(\bm{x})|$ or $\bm{y}=\bm{0}$. 
\end{itemize}
\end{proposition}

Now we prove 
Theorem \ref{thm:main} (i): 
$$
S_i\cap S_j=
\bigsqcup C, 
$$
where $C$ runs all chambers 
satisfying 
$\alpha_{i-1}(C)\alpha_i(C)<0, 
\alpha_{j-1}(C)\alpha_j(C)<0$ ($1\leq i<j\leq n$). 
Suppose that $\bm{x}+\ii\bm{y}\in S_i\cap S_j$ and 
$\bm{y}\neq\bm{0}$. 
Then by Proposition \ref{prop:sailsM} and 
Proposition \ref{prop:between}, we have 
$$
\theta_j\leq|\arg(\bm{y})|
\leq\theta_{j-1},\mbox{ and }
\theta_i\leq|\arg(\bm{y})|
\leq\theta_{i-1}. 
$$
This happens only when 
$\theta_{i-1}=\theta_i=\theta_{j-1}=\theta_j$, 
which means that 
$H_{i-1}, H_i, H_{j-1}$ and $H_j$ are parallel. 
However, since 
$\{\bm{x}\in\bR^2\mid\alpha_{i-1}\alpha_i<0\}$ 
and 
$\{\bm{x}\in\bR^2\mid\alpha_{j-1}\alpha_j<0\}$ 
are parallel strips, which do not intersect. 
This is a contradiction. Hence we have 
$\bm{y}=\bm{0}$, and $S_i\cap S_j$ is a union of 
chambers. (See Figure \ref{fig:transvint}.)

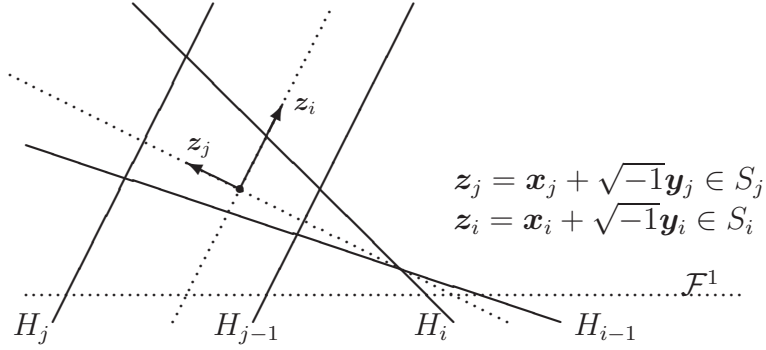
\begin{figure}[htbp]
\begin{picture}(100,120)(20,0)
\thicklines

\multiput(90,10)(3,0){90}{\circle*{1}}
\put(335,10){$\scF^1$}

\put(100,0){\line(1,2){60}}
\put(85,-5){$H_{j}$}
\put(175,0){\line(1,2){60}}
\put(160,-5){$H_{j-1}$}
\put(250,0){\line(-1,1){120}}
\put(235,-5){$H_{i}$}
\put(290,0){\line(-3,1){200}}
\put(295,-5){$H_{i-1}$}

\thinlines
\multiput(270,0)(-3,1.5){63}{\circle*{0.5}}

\multiput(145,0)(1.5,3){40}{\circle*{0.5}}

\thicklines
\put(170,50){\circle*{3}}
\put(170,50){\vector(-2,1){20}}
\put(150,65){\footnotesize $\bm{z}_j$}
\put(170,50){\vector(1,2){16}}
\put(190,80){\footnotesize $\bm{z}_i$}

\put(250,50){$\bm{z}_j=\bm{x}_j+\ii\bm{y}_j\in S_j$}
\put(250,35){$\bm{z}_i=\bm{x}_i+\ii\bm{y}_i\in S_i$}

\end{picture}
     \caption{$S_i$ and $S_j$ intersect transversely.}
\label{fig:transvint}
\end{figure}

\subsection{Contractibility of $S_i^\circ$}

Now we prove that $S_i^\circ=S_i\setminus\bigcup_{C\in\ch_2^\scF(\scA)}C$ 
is contractible. Let us denote $A_i:=S(\alpha_{i-1}, \alpha_i)\cap\scF^1$. 
Since $A_i$ is obviously contractible, therefore it suffices to 
construct a deformation retract onto $A_i$, that is, a family of 
continuous map 
$f_t:S_i^\circ\rightarrow S_i^\circ$ which satisfies 
$f_0=\id_{S_i^\circ}$, $f_1(S_i^\circ)=A_i$ and $f_t|_{A_i}=\id_{A_i}$. 

Define a continuous map $\rho:S_i^\circ\rightarrow A$, 
$\bm{x}+\ii\bm{y}\mapsto \rho(\bm{x}+\ii\bm{y})$ by 
\begin{itemize}
\item[(1)] 
if $\bm{y}\neq\bm{0}$, then 
$\rho(\bm{x}+\ii\bm{y})=A_i\cap(\bm{x}+\bR\cdot\bm{y})$, 
\item[(2)] 
if $H_i\cap H_{i-1}\neq\emptyset$ and $\bm{y}=\bm{0}$, then 
$\rho(\bm{x}+\ii\bm{y})=A_i\cap\overline{(\bm{x}\cdot H_i\cap H_{i-1})}$, 
where $\overline{(\bm{x}\cdot H_i\cap H_{i-1})}$ is the line passing 
through $\bm{x}$ and the intersection $H_i\cap H_{i-1}$, 
\item[(3)] 
if $H_i\cap H_{i-1}=\emptyset$ and $\bm{y}=\bm{0}$, then 
$\rho(\bm{x}+\ii\bm{y})=A_i\cap L_{\bm{x}}$, 
where $L_{\bm{x}}$ is the line passing 
through $\bm{x}$ and parallel to $H_i$. 
\end{itemize}
By Proposition \ref{prop:sailsM}, $\rho$ is a well-defined 
continuous map. Note that $\rho|_{A_i}=\id_{A_i}$.

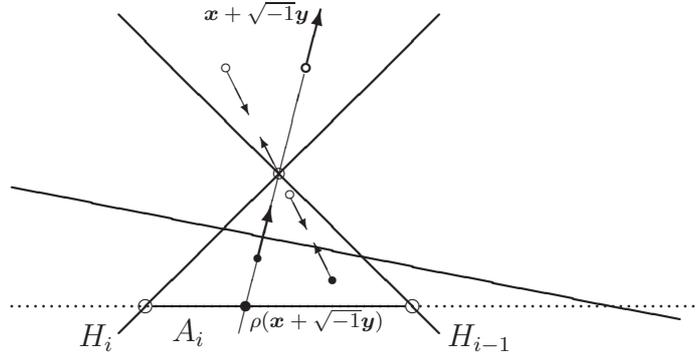
\begin{figure}[htbp]
\begin{picture}(100,120)(20,0)
\thicklines

\put(125,-5){$H_i$}
\put(140,0){\line(1,1){120}}

\put(260,0){\line(-1,1){120}}
\put(263,-5){$H_{i-1}$}

\put(100,55){\line(5,-1){250}}

\multiput(100,10)(3,0){87}{\circle*{1}}

\put(160,-3){$A_i$}
\thinlines
\put(153,10){\line(1,0){94}}
\put(150,10){\circle{5}}
\put(250,10){\circle{5}}

\thicklines
\put(192,28){\circle*{3}}
\put(192,28){\vector(1,4){5}}

\put(210,100){\circle{3}}
\put(210.5,102){\vector(1,4){5}}
\put(172,118){\scriptsize $\bm{x}+\ii\bm{y}$}

\thinlines
\put(185,0){\line(1,4){24.35}}

\put(180,100){\circle{3}}
\put(180.8,98.4){\vector(1,-2){8}}

\put(220,20){\circle*{3}}
\put(220,20){\vector(-1,2){7}}

\put(200,60){\vector(-1,2){7}}
\put(200,60){\circle{4}}

\put(204.5,51){\vector(1,-2){6}}
\put(204,52){\circle{3}}

\thicklines
\put(187.5,10){\circle*{4}}
\put(189,2){\scriptsize $\rho(\bm{x}+\ii\bm{y})$}

\end{picture}
     \caption{Deformation retract $\rho(\bm{x}+\ii\bm{y})$.}
\label{fig:contractS}
\end{figure}

Define 
$$
f_t(\bm{x}+\ii\bm{y})=
((1-t)\bm{x}+t\rho(\bm{x}+\ii\bm{y}))+
\ii(1-t)\bm{y}. 
$$
If $\bm{y}\neq\bm{0}$, then the real part 
$((1-t)\bm{x}+t\rho(\bm{x}+\ii\bm{y}))$ is 
on the line $\bm{x}+\bR\cdot\bm{y}$ and the imaginary 
part is nonzero provided $t\neq 1$. Hence 
$f_t(\bm{x}+\ii\bm{y})\in S_i^\circ$ 
(see also Lemma \ref{lem:unifmotsail}). 
If $\bm{y}=\bm{0}$, then $\bm{x}$ is contained 
in the chamber $C_i$. (Otherwise, $\bm{x}$ is contained 
in some chamber $C\in\ch_2^\scF(\scA)$ which does not 
intersects $\scF^1$.) 
Hence $f_t(\bm{x})\in C_i\subset S_i^\circ$. 
The map $f_t$ determines a deformation contraction of 
$S_i^\circ$ onto $A_i$. (See Figure \ref{fig:contractS}.)

\subsection{Contractibility of $U$}

We break the proof of the contractibility of 
$U=M(\scA)\setminus\bigcup_{i=1}^n S_i$ up into 
$n$ steps. 

\subsubsection{Filtration $U_k$}

\begin{definition}
Define $U_0=U$ and 
$$
U_k:=
\{\bm{z}=\bm{x}+\ii\bm{y}\in U
\mid
\alpha_1(\bm{x})\leq 0, \dots, 
\alpha_k(\bm{x})\leq 0\},  
$$
for $k=1, \dots, n$. 
\end{definition}
Obviously 
$$
U=U_0\supset U_1\supset \dots \supset U_n, 
$$
and 
\begin{proposition}
$U_n$ is star-shaped. In particular, $U_n$ is 
contractible. 
\end{proposition}
\noindent 
Therefore it is enough to construct a 
deformation retract $\rho_k:U_k\rightarrow U_{k+1}$ for 
$k=0, \dots, n-1$. 

\subsubsection{The case $k=0$}

First we construct a deformation retraction 
$\rho_0:U=U_0\rightarrow U_1=\{\bm{z}=\bm{x}+\ii
\bm{y}\mid\alpha_1(\bm{x})\leq 0\}$. 

Let $\bm{z}=\bm{x}+\ii\bm{y}\in U_0\setminus U_1$. Then, 
by definition, $\alpha_1(\bm{x})>0$. Recall Proposition 
\ref{prop:sails} 
that 
$$
S_1=\{\bm{x}+\ii\bm{y}
\mid
\alpha(\bm{x})>0\mbox{ and }
\bm{y} \mbox{ is either zero or parallel to } H_1\}. 
$$
Therefore $\bm{z}\notin\ S_1$ implies that the affine line 
$\bm{x}+\bR\cdot\bm{y}\subset\bR^2$ is 
not parallel to $H_1$, hence intersects $H_1$. 
Denote by $\tau(\bm{z})\in\bR$ the unique 
real number satisfying 
$\bm{x}+\tau(\bm{z})\bm{y}\in H_1$. 
Define the family of continuous map 
$f_t:U_0\rightarrow U_0$ 
($0\leq t\leq 1$) by 
$$
f_t(\bm{z})=
\left\{
\begin{array}{cl}
(\bm{x}+t\cdot\tau(\bm{z})\bm{y})+\ii
\bm{y}
&\mbox{ if }\bm{z}\in U_0\setminus U_1\\
\bm{x}+\ii\bm{y}
&\mbox{ if }\bm{z}\in U_1. 
\end{array}
\right.
$$
Then by Lemma \ref{lem:unifmotsail}, 
$f_t(\bm{z})\in U$. Hence $\rho_0=f_1:U_0\rightarrow U_1$ 
is a deformation retraction. (Figure \ref{fig:rho0}.)

\begin{figure}[htbp]
\begin{picture}(100,100)(20,0)
\thicklines

\multiput(90,10)(3,0){90}{\circle*{1}}
\put(335,10){$\scF^1$}

\put(100,0){\line(2,1){150}}
\put(85,-5){$H_{i}$}

\put(250,0){\line(-1,1){100}}
\put(235,-5){$H_{1}$}

\put(208,78){\circle{3}}
\thinlines
\put(207.5,80){\vector(-1,4){4.5}}
\thicklines
\put(215,90){$\bm{z}=\bm{x}+\ii\bm{y}\in U_0\setminus U_1$}

\thinlines
\put(206,78){\vector(-1,0){16}}
\thicklines

\multiput(220,30)(-0.7,2.8){17}{\circle*{1}}
\multiput(188,78)(-3,0){6}{\circle*{1}}

\thinlines
\put(220,30){\vector(-1,4){4.5}}
\thicklines

\put(220,30){\circle{3}}
\put(225,33){$\rho_0(\bm{z})$}

\thinlines
\put(172,78){\vector(-1,0){16}}
\thicklines

\put(172,78){\circle{3}}

\end{picture}
     \caption{$\rho_0:U_0\rightarrow U_1$.}
\label{fig:rho0}
\end{figure}
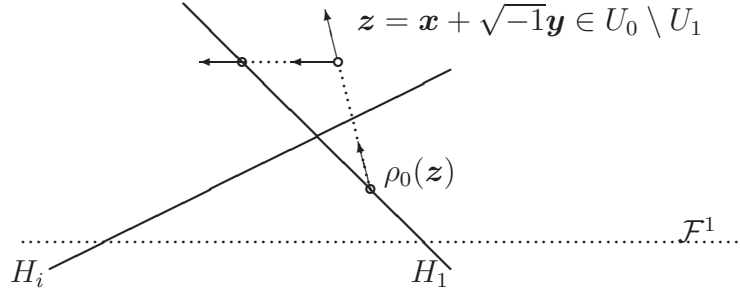

\subsubsection{The case that $H_k$ and $H_{k+1}$ are parallel}
\label{subsub:para}

Here, we assume that $H_k$ and $H_{k+1}$ are parallel 
($1\leq k\leq n-1$). 

\begin{definition}
\label{def:roof}
\begin{itemize}
\item[(1)] 
Define the closed subset 
$D_k\subset\bR^2$ by 
$$
D_k=
\{
\bm{x}\in\bR^2\mid
\alpha_1(\bm{x})\leq 0, 
\alpha_2(\bm{x})\leq 0, 
\dots, 
\alpha_k(\bm{x})\leq 0, \mbox{ and }
\alpha_{k+1}(\bm{x})\geq 0
\}. 
$$
\item[(2)] 
Denote the upper roof of $D_k$ by $R_k$. 
More precisely, $R_k$ is the closure of 
$\partial(D_k)\setminus(H_k\cup H_{k+1})$. 
\item[(3)] 
Suppose $\alpha_k(\bm{x})\leq0$ and $\alpha_{k+1}(\bm{x})\geq 0$. 
Then denote the line passing through $\bm{x}$ which is 
parallel to $H_k$ by $L_{\bm{x}}$. 
\end{itemize}
\end{definition}

\begin{figure}[htbp]
\begin{picture}(100,120)(20,0)
\thicklines

\multiput(90,10)(3,0){90}{\circle*{1}}
\put(349,10){$\scF^1$}

\thinlines
\put(104,0){\line(1,1){120}}
\put(85,-5){$H_{k+2}$}

\thicklines
\put(150,0){\line(0,1){120}}
\put(135,-5){$H_{k+1}$}

\put(250,0){\line(0,1){120}}
\put(235,-5){$H_{k}$}

\put(280,0){\line(-1,1){120}}
\put(283,-5){$H_{k-1}$}

\put(350,0){\line(-2,1){240}}
\put(324,-5){$H_{k-2}$}

\put(194,20){$D_k$}
\thinlines
\multiput(149.6,0)(0.1,0){8}{\line(0,1){100}}
\multiput(249.6,0)(0.1,0){8}{\line(0,1){30}}

\thicklines

\put(150,100){\circle*{5}}
\put(210,70){\circle*{5}}
\put(250,30){\circle*{5}}

\multiput(150,100)(2.4,-1.2){26}{\circle*{3}}
\multiput(210,70)(1.9,-1.9){21}{\circle*{3}}

\put(212,50){$R_k$}

\put(180,60){\circle*{3}}
\put(182,60){\footnotesize $\bm{x}$}

\thinlines
\put(180,0){\line(0,1){120}}
\put(182,112){\footnotesize $L_{\bm{x}}$}

\end{picture}
     \caption{$D_k$ and its roof $R_k$.}
\label{fig:roof}
\end{figure}
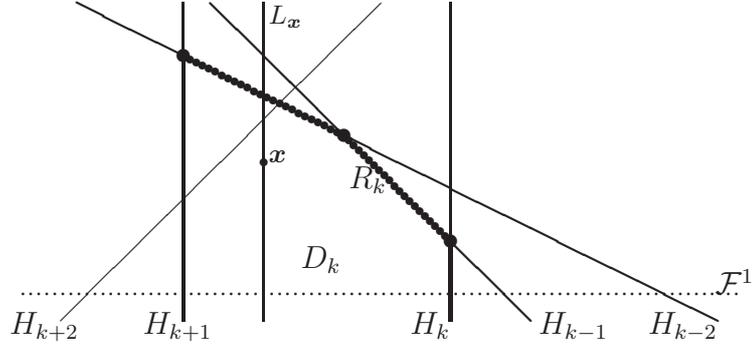

\begin{remark}
By definitions, if $\bm{x}+\ii\bm{y}\in 
U_k\setminus U_{k+1}$, then $\bm{x}\in D_k$. 
\end{remark}
The set 
$\{\bm{x}\in\bR^2\mid \alpha_k(\bm{x})\leq 0\leq 
\alpha_{k+1}(\bm{x})\}$
is a strip with boundaries $H_k$ and $H_{k+1}$. 
We can define a deformation retract of this 
strip to $D_k$ by 
$$
\pr_k(\bm{x})=
\left\{
\begin{array}{cl}
R_k\cap L_{\bm{x}}&\mbox{ if }\bm{x}\notin D_k,\\
\bm{x}&\mbox{ if }\bm{x}\in D_k.
\end{array}
\right.
$$
Suppose $\bm{z}=\bm{x}+\ii\bm{y}\in U_k\setminus U_{k+1}$. 
Since $\bm{z}\notin S_k$, $\bm{y}$ is neither zero nor 
parallel to $H_{k+1}$. Hence there exists a unique 
real number $\tau(\bm{x}, \bm{y})\in\bR$ such that 
$\bm{x}+\tau(\bm{x}, \bm{y})\cdot\bm{y}\in H_{k+1}$. 

Define the family of continuous map $F_t$ ($0\leq t\leq 1$) by 
\begin{equation}
F_t(\bm{x}+\ii\bm{y})=
\pr_k(\bm{x}+t\cdot\tau(\bm{x}, \bm{y})\cdot\bm{y})+
\ii\bm{y}
\end{equation}
for $\bm{x}+\ii\bm{y}\in U_k$ with 
$\bm{x}\in D_k$. (Figure \ref{fig:retractpara}.)

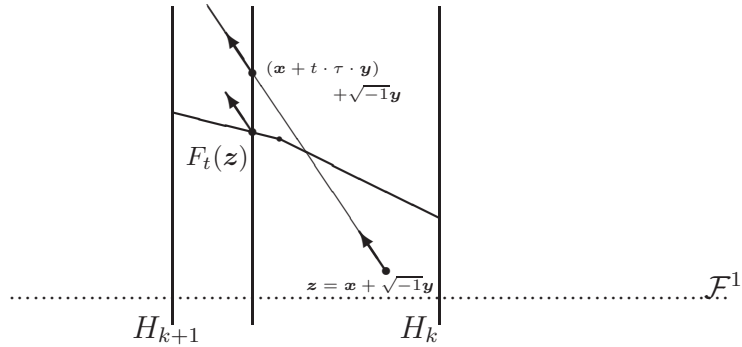
\begin{figure}[htbp]
\begin{picture}(100,120)(20,0)
\thicklines

\multiput(90,10)(3,0){90}{\circle*{1}}
\put(349,10){$\scF^1$}


\thicklines
\put(150,0){\line(0,1){120}}
\put(135,-5){$H_{k+1}$}

\put(250,0){\line(0,1){120}}
\put(235,-5){$H_{k}$}

\put(190,70){\circle*{2}}

\put(250,40){\line(-2,1){60}}

\put(190,70){\line(-4,1){40}}

\put(230,20){\circle*{3}}
\put(230,20){\vector(-2,3){10}}
\put(200,13){\tiny $\bm{z}=\bm{x}+\ii\bm{y}$}

\thinlines
\put(230,20){\line(-2,3){67}}

\thicklines
\put(180,95){\circle*{3}}
\put(180,95){\vector(-2,3){10}}

\put(185,95){\tiny $(\bm{x}+t\cdot\tau\cdot\bm{y})$}
\put(210,85){\tiny $+\ii\bm{y}$}

\thinlines
\put(180,0){\line(0,1){120}}

\thicklines
\put(180,72.5){\circle*{3}}
\put(180,72.5){\vector(-2,3){10}}

\put(155,60){\footnotesize $F_t(\bm{z})$}

\end{picture}
     \caption{Retraction $F_t$.}
\label{fig:retractpara}
\end{figure}

\begin{proposition}
\label{prop:extendpara}
Let us extend the above $F_t$ by 
$$
F_t(\bm{x}+\ii\bm{y})=
\left\{
\begin{array}{cl}
F_t(\bm{z}) \mbox{ (as above)}&
\mbox{ if }\bm{z}\in U_k\setminus U_{k+1}, \\
\bm{z}& \mbox{ if } \bm{z}\in U_{k+1}, 
\end{array}
\right.
$$
Then $F_t(\bm{z})\in U_k$ for any 
$\bm{z}\in U_k$ and hence 
$F_1$ determines a deformation retract 
$U_k\rightarrow U_{k+1}$. 
\end{proposition}

\begin{proof}
Let $\bm{z}=\bm{x}+\ii\bm{y}\in U_k$ and 
$F_t(\bm{z})=\bm{z'}=\bm{x'}+\ii\bm{y}$. 
Suppose that $\bm{z'}\notin U_k$. 
If $\bm{x}+t\cdot\tau\cdot\bm{y}\in D_k$, 
then $F_t(\bm{z})=(\bm{x}+t\cdot\tau\cdot\bm{y})+\ii\bm{y}$. 
By Lemma \ref{lem:unifmotsail} 
$F_t(\bm{z})\in M(\scA)$, 
hence contained in $U_k$. 
Thus we may assume that 
$\bm{x}+t\cdot\tau\cdot\bm{y}\notin D_k$ and 
$\bm{x'}\in R_k$. Furthermore, we may assume that 
$\bm{y}\in T_{\bm{x'}}\bR^2$ is contained in 
a line $H_j\subset T_{\bm{x'}}$ with 
$\bm{x'}\in H_j$ for some $1\leq j<k$. 
Then $\bm{x}+t\cdot\tau\cdot\bm{y}$ must be 
contained in the domain $\{\alpha_j>0\}$. 
However, this contradicts $\bm{x}\in\{\alpha_j\leq 0\}$ 
and the fact that $\bm{y}$ is parallel to $H_j$. 
Hence $F_t(\bm{z})\in U_k$. 
\end{proof}

\subsubsection{$LQ$-curves}

The remaining case is the construction of 
deformation retract $U_k\rightarrow U_{k+1}$ when 
$H_k$ and $H_{k+1}$ are not parallel. 
The idea is similar to the previous case, however, 
it requires more technicality. 

An Linear-Quadric curve on $\bR^2$ is, roughly speaking, 
a $C^1$-curve 
which is linear when $x_1\leq 1$ and quadric 
when $x_1\geq 1$. The precise definition is as follows. 

\begin{definition}
\label{def:lq}
An {\em LQ-curve} (Linear-Quadric-curve) $C$ on 
the real plane $\bR^2$ is either a vertical line 
$C=\{(x_1, x_2)\in\bR^2\mid x_1=t\}$ or the 
graph $\{(x, f(x))\mid x\in \bR\}$ of a 
$C^1$-function $f(x)$ such that 
$$
f(x_1)=
\left\{
\begin{array}{cl}
ax_1+b&\mbox{ for }x_1\leq 1,\\
cx_1^2+dx_1 &\mbox{ for }x_1\geq 1, 
\end{array}
\right.
$$
where $t, a, b, c, d\in\bR$. 
\end{definition}

\begin{remark}
\begin{itemize}
\item[(1)] Since $f(x_1)$ is $C^1$ at $x_1=1$, 
$f(x_1)$ should have the following expression. 
\begin{equation}
\label{eq:lqexpress}
f(x_1)=
\left\{
\begin{array}{cl}
ax_1+b&\mbox{ for }x_1\leq 1,\\
-bx_1^2+(a+2b)x_1 &\mbox{ for }x_1\geq 1. 
\end{array}
\right.
\end{equation}
\item[(2)] $f(x_1)$ and the derivative $f'(x_1)$ 
for some $x_1\in\bR$ determines the unique LQ-curve. 
\end{itemize}
\end{remark}

Let 
$\bm{x}\in\bR^2$ 
be a point in the positive quadrant and 
$\bm{y}\in T_{\bm{x}}\bR^2\setminus\{\bm{0}\}$ a 
nonzero tangent vector. 
Then there exists a unique 
$C^1$-map $X_{\bm{x},\bm{y}}:\bR\rightarrow\bR^2$ 
such that 
\begin{itemize}
\item 
$X(0)=\bm{x}, \dot{X}(0)=\bm{y}$, 
\item 
$\{X(t)\mid t\in\bR\}\subset\bR^2$ is an LQ-curve. 
\item 
$|\dot{X}(t)|=|\bm{y}|$. 
\end{itemize}
Roughly speaking, $X(t)$ is a motion along an 
LQ-curve with constant velocity. 
$X_{\bm{x},\bm{y}}(t)$ is continuous with 
respect to $\bm{x}, \bm{y}$ and $t$. 

In the remainder of this section, 
we assume $\bm{x}\in(\bR_{\geq 0})^2$ and 
$\bm{x}\neq\bm{0}$. 
Then $0\leq \arg\bm{x}\leq \frac{\pi}{2}$. 
We also assume that $\bm{y}\notin\bR\cdot\bm{x}$. 
We call $\bm{y}$ positive (resp. negative) 
if $\arg\bm{x}<\arg\bm{y}<\arg\bm{x}+\pi$ 
(resp. $\arg\bm{x}-\pi<\arg\bm{y}<\arg\bm{x}$). 
It is easily seen if $\bm{y}$ is positive (resp. negative), 
then $\arg X_{\bm{x},\bm{y}}(t)$ is increasing 
(resp. decreasing) in $t$. 

\begin{lemma}
\label{lem:lqhit}
Let $\bm{x}$ and $\bm{y}$ as above. 
Then the LQ-curve 
$X_{\bm{x}, \bm{y}}(t)$ intersects the 
positive $x_1$-axis $\{(x_1, 0)\mid x_1>0\}$ 
exactly once. 
\end{lemma}

\begin{proof}
If $\bm{y}$ is vertical, the assertion holds obviously. 
Assume that $\bm{y}$ is not vertical. 
We use the expression (\ref{eq:lqexpress}). 
From the assumption that 
$\bm{y}\notin\bR\cdot\bm{x}$, 
$b\neq 0$. Suppose $b>0$. 
The quadric equation $-bx^2+(a+2b)x=0$ 
has the solution $x=\frac{a+2b}{b}$ (and $x=0$). 
If $\frac{a+2b}{b}>1$, then we have $a+b=f(1)>0$. 
Since $f(0)=b>0$, $f(t)\neq 0$ for $0\leq t\leq 1$. 
Hence $x=\frac{a+2b}{b}$ is the unique solution. 
$\frac{a+2b}{b}=1$ is equivalent to say $a+b=0$. 
Hence $x_1=1$ is the unique solution of $f(x_1)=0$. 
$\frac{a+2b}{b}<1$ implies that $a+b<0$. 
Since $f(0)>0>f(1)$, there exists the unique 
solution $f(t)=0$ with $0\leq t\leq 1$. 
The case $b<0$ is similar. 
\end{proof}

\begin{definition}
Let $\bm{x}$ and $\bm{y}$ be as above. 
Denote by $\tau=\tau(\bm{x}, \bm{y})$ the unique 
real number such that $X_{\bm{x}, \bm{y}}(\tau)\in
\{(x_1, 0)\mid x_1> 0\}$. 
(Figure \ref{fig:lqcurves}.) 
\end{definition}

\begin{remark}
$\tau(\bm{x}, \bm{y})$ is continuous on 
$\{(\bm{x}, \bm{y})\mid\bm{x}
\in(\bR_{\geq 0})^2\setminus\{\bm{0}\}, 
\bm{y}\notin\bR\cdot\bm{x}\}$. 
\end{remark}

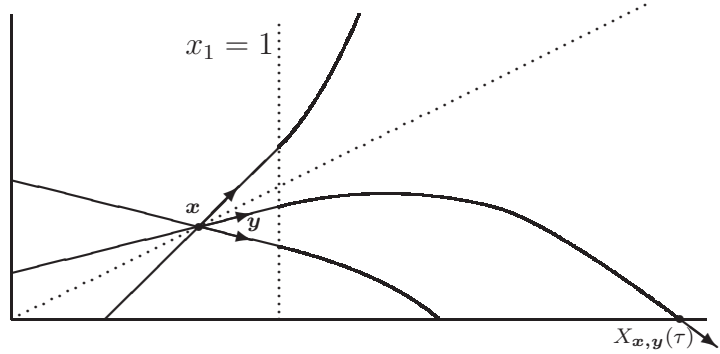
\begin{figure}[htbp]
\begin{picture}(100,120)(20,-5)
\thicklines

\put(100,0){\line(1,0){260}}
\put(100,0){\line(0,1){115}}

\multiput(200,0)(0,3){38}{\circle*{1}}

\multiput(100,0)(3,1.5){80}{\circle*{1}}

\put(165,100){$x_1=1$}

\put(170,35){\circle*{3}}
\put(165,40){\scriptsize $\bm{x}$}
\put(188,35){\scriptsize $\bm{y}$}

\put(170,35){\vector(1,1){15}}
\put(135,0){\line(1,1){65}}
\qbezier(200,65)(215,80)(230,115)

\put(170,35){\vector(4,1){20}}
\put(100,17.5){\line(4,1){100}}
\qbezier(200,42.5)(220,47.5)(240,47.5)
\qbezier(240,47.5)(260,47.5)(280,42.5)
\qbezier(280,42.5)(300,37.5)(350,0)
\put(350,0){\vector(4,-3){14}}
\put(350,0){\circle*{3}}
\put(325,-9){\scriptsize $X_{\bm{x, y}}(\tau)$}

\put(170,35){\vector(4,-1){20}}
\put(100,52.5){\line(4,-1){100}}
\qbezier(200,27.5)(240,17.5)(260,0)

\end{picture}
     \caption{LQ-curves $X_{\bm{x}, \bm{y}}(t)$.}
\label{fig:lqcurves}
\end{figure}

\subsubsection{The case that $H_k$ and $H_{k+1}$ are not parallel}

Next we assume that $H_k$ and $H_{k+1}$ are not parallel, 
and constructing a deformation retraction 
$\rho_k:U_k\rightarrow U_{k+1}$. 
The idea is similar to the parallel case 
(\S\ref{subsub:para}). However we need LQ-curves 
to construct the retraction. 

Here we choose 
coordinates $x_1, x_2$ such that 
$\alpha_k=-x_1$, 
$\alpha_{k+1}=x_2$ and 
$\scF^1=\{x_1+x_2=1\}$. 
Recall Definition \ref{def:roof} that 
$D_k\subset\bR^2$ is defined by 
$$
D_k=
\{
\bm{x}\in\bR^2\mid
\alpha_1(\bm{x})\leq 0, 
\alpha_2(\bm{x})\leq 0, 
\dots, 
\alpha_k(\bm{x})\leq 0, \mbox{ and }
\alpha_{k+1}(\bm{x})\geq 0
\}, 
$$
and the roof $R_k$ is defined as 
the closure of 
$\partial(D_k)\setminus(H_k\cup H_{k+1})$. 

\begin{definition}
Suppose $\bm{x}\neq\bm{0}$. 
Then denote the line passing through $\bm{x}$ and 
the intersection 
$\{\bm{0}\}=H_k\cap H_{k+1}$ by $L_{\bm{x}}$. 
\end{definition}

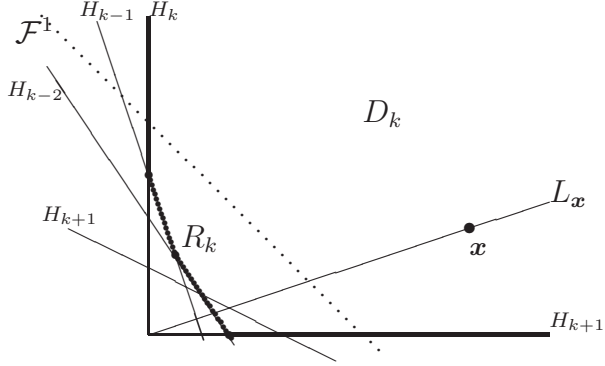
\begin{figure}[htbp]
\begin{picture}(100,130)(20,0)
\thicklines

\put(100,0){\line(1,0){150}}
\put(100,0){\line(0,1){120}}

\put(50,110){$\scF^1$}
\multiput(60,120)(3,-3){43}{\circle*{1}}

\thinlines

\put(250,3){\scriptsize $H_{k+1}$}

\put(100,120){\scriptsize $H_k$}

\put(74,120){\scriptsize $H_{k-1}$}
\put(80.666,118){\line(1,-3){40}}

\put(47,90){\scriptsize $H_{k-2}$}
\put(132,-4){\line(-2,3){70}}

\put(100.5,60){\line(0,1){60}}
\put(99.5,60){\line(0,1){60}}

\put(130,0.5){\line(1,0){120}}
\put(130,-0.5){\line(1,0){120}}

\put(112,33){$R_k$}
\put(100,60){\circle*{3}}
\multiput(100,60)(0.6,-1.8){17}{\circle*{1.5}}
\put(110,30){\circle*{3}}
\multiput(110,30)(1,-1.5){22}{\circle*{1.5}}
\put(130,0){\circle*{3}}

\put(60,43){\scriptsize $H_{k+1}$}
\put(70,40){\line(2,-1){100}}

\put(250,50){$L_{\bm{x}}$}
\put(220,30){\footnotesize $\bm{x}$}
\put(220,40){\circle*{4}}
\put(100,0){\line(3,1){150}}

\put(180,80){$D_k$}

\end{picture}
     \caption{$D_k$, $R_k$ and $L_{\bm{x}}$.}
\label{fig:nonparalDk}
\end{figure}

We can define a 
deformation retract of 
$(\bR_{\geq 0})^2\setminus\{\bm{0}\}$ 
to $D_k$ by 
$$
\pr_k(\bm{x})=
\left\{
\begin{array}{cl}
R_k\cap L_{\bm{x}}&\mbox{ if }\bm{x}\notin D_k,\\
\bm{x}&\mbox{ if }\bm{x}\in D_k.
\end{array}
\right.
$$

Suppose $\bm{z}=\bm{x}+\ii\bm{y}\in U_k\setminus U_{k+1}$. 
Since $\bm{z}\notin S_k$, $\bm{y}\notin\bR\cdot\bm{x}$ 
(Proposition \ref{prop:sails}). 
Hence there exists a unique 
real number $\tau=\tau(\bm{x}, \bm{y})\in\bR$ such that 
$X_{\bm{x}, \bm{y}}(\tau)\in H_{k+1}$. 

Define the family of continuous map $F_t$ ($0\leq t\leq 1$) by 
\begin{equation}
F_t(\bm{x}+\ii\bm{y})=
\pr_k(X_{\bm{x}, \bm{y}}(t\cdot \tau(\bm{x}, \bm{y})))+
\ii \dot{X}_{\bm{x}, \bm{y}}(t\cdot \tau(\bm{x}, \bm{y})), 
\end{equation}
for $\bm{x}+\ii\bm{y}\in U_k$ with 
$\bm{x}\in D_k$. (Figure \ref{fig:retractnonpara}.)

\begin{figure}[htbp]
\begin{picture}(100,150)(20,0)
\thicklines

\put(40,0){\line(1,0){320}}
\put(50,-10){\line(0,1){160}}

\put(23,144){\footnotesize $\scF^1$}
\multiput(40,150)(3,-3){53}{\circle*{1}}


\put(360,-7){\scriptsize $H_{k+1}$}

\put(52,147){\scriptsize $H_k$}



\put(50,120){\circle*{2}}
\put(50,120){\line(1,-3){20}}
\put(70,60){\circle*{2}}
\put(70,65){\footnotesize $R_k$}
\put(70,60){\line(5,-6){50}}
\put(120,0){\circle*{2}}

\put(110,140){\circle*{3}}
\put(110,140){\vector(-1,-4){6}}
\thinlines
\put(110,140){\line(-1,-4){37}}
\put(95,120){\footnotesize $\bm{z}$}

\thicklines
\put(110,140){\vector(1,0){20}}
\put(135,142){\footnotesize $\bm{z'}$}

\thicklines
\put(80,20){\circle*{3}}
\put(80,20){\vector(-1,-4){6}}

\thinlines
\put(50,0){\line(3,2){55}}

\thicklines
\put(95,30){\circle*{3}}
\put(95,30){\vector(-1,-4){6}}
\put(102,25){\footnotesize $F_t(\bm{z})$}

\put(120,0){\circle*{3}}
\put(120,0){\vector(-1,-4){6}}
\put(120,-9){\footnotesize $F_1(\bm{z})$}

\multiput(110,140)(3,0){27}{\circle*{1}}
\put(190,140){\circle*{1}}
\put(193,139.84){\circle*{1}}
\put(196,139.378){\circle*{1}}
\put(199,138.6){\circle*{1}}
\put(202,137.511){\circle*{1}}
\put(205,136.111){\circle*{1}}
\put(208,134.4){\circle*{1}}
\put(211,132.378){\circle*{1}}
\put(214,130.044){\circle*{1}}
\put(217,127.4){\circle*{1}}
\put(220,124.444){\circle*{1}}
\put(223,121.178){\circle*{1}}
\put(226,117.6){\circle*{1}}
\put(229,113.711){\circle*{1}}
\put(232,109.511){\circle*{1}}
\put(235,105){\circle*{1}}
\put(238,100.178){\circle*{1}}
\put(241,95.0444){\circle*{1}}
\put(244,89.6){\circle*{1}}
\put(247,83.844){\circle*{1}}
\put(250,77.777){\circle*{1}}
\put(253,71.4){\circle*{1}}
\put(256,64.7111){\circle*{1}}
\put(259,57.711){\circle*{1}}
\put(262,50.4){\circle*{1}}
\put(265,42.777){\circle*{1}}
\put(268,34.844){\circle*{1}}
\put(271,26.6){\circle*{1}}
\put(274,18.0444){\circle*{1}}
\put(277,9.1777){\circle*{1}}
\put(280,0){\circle*{1}}

\put(282,5){\footnotesize $F_1(\bm{z'})=
{X}_{\bm{x}, \bm{y}}(\tau)+\ii
\dot{X}_{\bm{x}, \bm{y}}(\tau)$}
\put(280,0){\circle*{3}}
\put(280,0){\vector(1,-2){7}}

\end{picture}
     \caption{$F_t(\bm{z})$.}
\label{fig:retractnonpara}
\end{figure}

The next proposition completes the 
proof of the main result, which 
is proved in a similar way to the proof of 
Proposition \ref{prop:extendpara}. 

\begin{proposition}
\label{prop:extendnonpara}
Let us extend the above $F_t$ by 
$$
F_t(\bm{x}+\ii\bm{y})=
\left\{
\begin{array}{cl}
F_t(\bm{z}) \mbox{ (as above)}&
\mbox{ if }\bm{z}\in U_k\setminus U_{k+1}, \\
\bm{z}& \mbox{ if } \bm{z}\in U_{k+1}, 
\end{array}
\right.
$$
Then $F_t(\bm{z})\in U_k$ for any 
$\bm{z}\in U_k$ and hence 
$F_1$ determines a deformation retract 
$U_k\rightarrow U_{k+1}$. 
\end{proposition}

\medskip

\noindent
{\bf Acknowledgement.} 
A part of this work was done while the author 
visited Universit\`a di Pisa and 
Centro di Ricerca Matematica Ennio De Giorgi. 
The author appreciates their supports and hospitality. 
The author greatefully acknowledges the financial support 
of JSPS. 

\noindent
Masahiko Yoshinaga

Department of Mathematics, Kyoto University, 
Sakyo-ku, Kyoto, 606-8502, Japan. 
Email: mhyo@math.kyoto-u.ac.jp

\end{document}